\documentclass[12pt,reqno]{amsart}
\advance \topmargin by -\headheight
\advance \topmargin by -\headsep
\evensidemargin \oddsidemargin
\usepackage{amsmath}
\usepackage{amsfonts}
\usepackage{stmaryrd}
\usepackage{amssymb}
\usepackage{amsthm}
\usepackage{csquotes}
\usepackage[toc,page]{appendix} 
\usepackage{mathrsfs}
\usepackage{dsfont}
\usepackage{bm}
\usepackage{cite}
\usepackage{soul}
\usepackage{setspace}
\usepackage{breqn}
\usepackage{tikz}
\usetikzlibrary{calc}
\usepackage{subfig}
\usepackage{subcaption}  
\numberwithin{equation}{section}

\newcommand{\n}[1]{\|{#1}\|}

\newcommand{\R}{\mathbb{R}}
\newcommand{\diam}{\text{diam}}
\newcommand{\diag}{\text{diag}}
\newcommand{\rank}{\text{rank}}
\newtheorem{theorem}{Theorem}[section]
\newtheorem{corollary}{Corollary}[theorem]
\newtheorem{lemma}[theorem]{Lemma}

\newtheorem{definition}[theorem]{Definition}
\newtheorem{example}[theorem]{Example}

\usepackage{mathtools}
\usepackage{hyperref}

\title[Additive growth amongst images of linearly independent analytic functions]{Additive growth amongst images of linearly independent analytic functions}

\author{Samuel Mansfield}
\address{Department of Mathematics, University of Bristol, Bristol, BS8 1UG, UK}
\email{sam.mansfield@bristol.ac.uk}

\date{}

\begin{document}

\begin{abstract}
    Let $\mathcal{F}$ be a set of $n$ real analytic functions with linearly independent derivatives restricted to a compact interval $I$. We show that for any finite set $A \subset I$, there is a function $f \in \mathcal
    F$ that satisfies
    $$|2^{n-1}f(A)-(2^{n-1}-1)f(A)|\gg_{\mathcal{F},I} |A|^{\phi(n)},$$
    where $\phi:\mathbb{N} \to \mathbb{R}$ satisfies the recursive formula
    $$\phi(1)=1, \quad \phi(n)=1+\frac{1}{1+\frac{1}{\phi(n-1)}} \quad \text{for } n\geq 2.$$
    The above result allows us to prove the bound
    $$|2^nf(A-A)-(2^n-1)f(A-A)| \gg_{f,n,I} |A|^{1+\phi(n)}$$
    where $f$ is an analytic function for which any $n$ distinct non-trivial discrete derivatives of $f'$ are linearly independent. This condition is satisfied, for instance, by any polynomial function of degree $m \geq n+1$. We also check this condition for the function $\arctan(e^x)$ with $n=3$, allowing us to improve upon a recent bound on the additive growth of the set of angles in a Cartesian product due to Roche-Newton.
    % This paper gives a generic condition for a set of real valued functions $\mathcal{F}$ on an interval $I\subseteq \R$ to contain, for any $\epsilon>0$, at least one function $f \in \mathcal{F}$ satisfying, with $s \leq k(\epsilon)\in \mathbb{N}$, the bound
    
    % for any finite subset $A \subset I$.
    % This condition is, in particular, satisfied by any set of analytic functions with linearly independent derivatives. 
\end{abstract}

\maketitle

% \section{Things that will go in the prelim chapter}

% \begin{theorem}\label{ENR} Let $A \subseteq \R$ be a finite set and let $f$ be a convex function 
%     $$|A+A||f(A)+f(A)| \gg |A|^{5/2}$$
% \end{theorem}
% Elekes, Nathanson and Ruzsa \cite{ENR}.

% A finite set $A =\{a_1<\dots<a_{|A|}\} \subseteq \R$ is called convex if, for all $i=1,\dots, |A|-1$ the sequence
% $$d_i\coloneqq a_{i+1}-a_{i}$$
% is strictly increasing.

% % squares: $|A+A| \ll k^2/\log^{1/2}(k)$.

% If convex sets cannot have additive structure, then convex functions ought to destroy additive structure when it is present. 
\section{Introduction}
We say that a subset $A=\{a_1<\dots<a_{|A|}\} \subset \R$ is convex if the sequence of consecutive differences 
$$d_i\coloneqq a_{i+1}-a_i \quad \text{for } 1 \leq i \leq |A|-1$$ is strictly increasing. Initiated by Erd\H{o}s, the investigation into the additive growth of convex sets of real numbers has accumulated significant interest over the past 40 years. This `additive growth' is typically quantified by the growth of the set $A$ under (repeated) applications of the \textit{sum set} and \textit{difference set} operations 
\begin{align*}
    A+A&\coloneqq\{a_1+a_2:a_1,a_2 \in A\},\\   A-A &\coloneqq \{a_1-a_2: a_1, a_2 \in A\},\\
     sA&:=\{a_1+\dots+a_s:a_1,\dots,a_s \in A\}.
\end{align*}
It is conjectured that for any finite convex set $A \subset \R$ the following bound ought to hold for all $\epsilon>0$\footnote{We write $Y \gg_{z} X$ to mean $X \geq C({z}) Y$ where $C(z)$ is some positive constant depending only on the parameter $z$}.
$$|A+A|\gg_\epsilon|A|^{2-\epsilon}.$$
The first non-trivial lower bound for $|A+A|$ is due Hegyv\'{a}ri \cite{hegyvari1986consecutive}, with the current records, held by Rudnev-Stevens \cite{rudnev2022update} for sum sets and Bloom \cite{bloom} for difference sets. These records are, for any $\epsilon>0$, 
\begin{equation*}\label{records}
    |A+A|\gg_\epsilon |A|^{30/19-\epsilon} \qquad \text{and} \qquad |A-A| \gg_\epsilon|A|^{6681/4175-\epsilon}.
\end{equation*}
% The obstruction that convexity poses to additive structure motivates the introduction of \textit{convex functions}.

% \begin{definition} Let $I$ be an interval. A function $f:I\to \R$ is convex if both $f$ and its derivative $f'$ are strictly increasing (or decreasing).    
% \end{definition}
For the longer sum $A+A-A$, a simple but clever \textit{squeezing} argument due to Ruzsa-Shakan-Solymosi-Szemer\'{e}di in \cite{shakanetc}, building upon ideas of Garaev \cite{garaev2000lower}, proves the following bound.
\begin{theorem}[Ruzsa-Shakan-Solymosi-Szemer\'edi \cite{shakanetc}]\label{basicsqueeze}
    Let $A \subset \R$ be a finite convex set, then 
    $$|A+A-A| \gg |A|^2.$$
\end{theorem}
% Sharpness can be demonstrated with the example $A=\{1^2,\dots,N^2\}.$
This intuition can be extended---sets that are somehow `more convex' should grow further under repeated addition. To quantify this idea, Hanson--Roche-Newton--Rudnev \cite{MishaBrandonOlly} introduce the notion of \textit{higher convexity}.
\begin{definition}\label{convexset} Let $A=\{a_1<\dots<a_{|A|}\}$ of real numbers and consider the set of consecutive differences $D\coloneqq \{a_{i+1}-a_i: 1 \leq i \leq |A|\}.$ We say, inductively, that any set written in increasing order is $0$-convex and that $A$ is $k$-convex for $k \geq 1$ if $D$ is a $(k-1)$-convex set. Note that a set is convex if and only if it is $1$-convex.
\end{definition} 
\textit{Remark}: One may replace the word `increasing' with `decreasing' in the above to get the definition of a \textit{concave} set. Note that since additive growth is invariant under reflections, all stated results hold with `convex set' replaced by `concave set'.  
    
With the concept of higher convexity in hand, the aforementioned squeezing techniques are pushed further to elicit larger growth for sets with higher-order convexity. In particular, Hanson---Roche-Newton---Rudnev \cite{MishaBrandonOlly} prove the following estimate.
\begin{theorem}[Hanson---Roche-Newton---Rudnev \cite{MishaBrandonOlly}]\label{longsumconvexset}
    Let $k \in \mathbb{N}$ and $A \subset \R$ be a finite $k$-convex set. Then 
    \begin{equation*}
          |2^k A - (2^k-1)A| \gg_k |A|^{k+1}. 
    \end{equation*}
\end{theorem}
% Theorems \ref{basicsqueeze} and \ref{longsumconvexset} capture the intuition that the more convex a set is, the more additive growth it displays. 
\subsection{Growth for sum sets of convex functions}
% Let $I$ be an interval. We say that a strictly monotone continuously differentiable function $f: I \to \R$ is \textit{convex} if its derivative $f'$ is also strictly monotone. Note that this is not the standard definition of convexity that the reader may be familiar with from other areas. In more common terms, we require \textit{strict}\footnote{For instance, the function $f(x)=x$ is not convex by our definition.} convexity and do not distinguish between convexity and concavity. Indeed, neither does additive growth:
 A finite subset $A \subset \R$ is convex (or concave) if and only if it is the image of the arithmetic progression $[N]=\{1,\dots,N\}$ under some strictly monotone and strictly convex (or concave) function $f: [1,N] \to \R$. Since arithmetic progressions are the archetype of sets that do not exhibit additive growth, one expects, for strictly convex or concave $f$, the sets $f(A)$ and $A$ to have inversely related additive growth. That is, $f(A)+f(A)$ should be large when $A+A$ is small and vice versa. The first result of this form is due to Elekes-Nathanson-Ruzsa \cite{ENR2000}, with the current records held by Stevens-Warren \cite{stevens2022sum} and Bloom \cite{bloom} for sums and Bloom \cite{bloom} for differences. These bounds are, for all $\epsilon>0$,
    \begin{equation*}
    |A+A||f(A)+f(A)| \gg_\epsilon |A|^{49/19-\epsilon} \qquad \text{and} \qquad |A-A||f(A)-f(A)| \gg_\epsilon |A|^{5428/4175-\epsilon}.
    \end{equation*}  
% From the above, if $A$ is additively structured in the sense that $A+A$ is small, then we know that $f(A)$ must lack additive structure in the sense that $|f(A)+f(A)|$ is large. Taking $f$ to be the convex function (by our definition) $f(x)=\log(x)$, results of the above form can be seen as a generalisation of the sum-product phenomenon. 

% First studied by Erd\H{o}s-Szemer\'{e}di \cite{ES1983}, the sum-product conjecture asserts that given any $\epsilon >0$ and any finite set $A \subset \mathbb{R}$, one ought to have
% \[\max\{|A+A|,|AA|\} \gg_{\epsilon} |A|^{2-\epsilon},\]
% where $AA = \{ a_1a_2 : a_1,a_2 \in A\}.$which case, we would expect $|sA|$ to be very large. This conjecture remains open, despite a significant body of work existing on it. The best known bounds are recorded in the work of Rudnev--Stevens \cite{RS2022} for the $s=2$ case, and in a paper by P\'{a}lv\"{o}lgyi--Zhelezov \cite{PZ2020} for the case when $s$ is large. In particular, building upon the breakthrough work of Bourgain--Chang \cite{BC2004}, the latter paper shows that for any $s \in \mathbb{N}$ and for any finite $A \subseteq \mathbb{Z}$, one has
% \[ |sA| + |A^{(s)}| \gg_{s} |A|^{c \log s/ \log \log s},  \]
% for some absolute constant $c>0$.
% Taking $f$ in the above theorem to be the `convex' (concave) function $f(x)=\log(x)$, we immediately get a sum-product bound of the form \eqref{sumproddelta} with $\delta=1/4$. For this reason, the study of convexity is also motivated by seeking sum-product results over the reals. See Chapter \ref{Convexity} for more details on convexity.

Similar to the situation with convex sets, `higher convexity' affords us more growth with longer sums. For functions, higher convexity is described by the following definition.
\begin{definition}
 Let $I \subseteq \R$ be an interval. We say a function $f:I \to \R$ is $0$-convex if it is strictly increasing on $I$. Further, for $k \geq 1$, we say that a function $f \in C^k(I)$ is $k$-convex if $f$ and each of its first $k$ derivatives are strictly increasing on $I$. 
\end{definition}
\textit{Remark:} Note that, by the above definition, \textit{convexity} and \textit{1-convexity} are not the same. A convex function may be not be strictly convex and it could be decreasing on some subinterval, but a 1-convex function is necessarily strictly increasing and strictly convex. This terminological distinction will be helpful later. We also remark that all results that hold for $k$-convex functions hold in the more general case that each of the derivatives $f^{(0)},f^{(1)},\dots,f^{(k)}$ are strictly monotone.    
% \textit{Remark:} Note that this is not the standard definition of convexity for functions that the reader may be familiar with from other areas. In more common terms, we require \textit{strict} convexity and do not distinguish between convexity and concavity. Indeed, neither does additive growth.
% A set $A$ is $k$-convex if and only if $A$ is the image of the interval the discrete 

% If sets with more convexity must have less additive structure, then functions that are `more' convex should destroy `more' additive structure.

% In this paper we will often be working with images of sets of real numbers $A$ under functions $f$ which we call $k$-convex with no mention of their domain. One should assume that these functions are $k$-convex on an interval containing $A$. For simplicity we will also assume beyond just monotonicity that the first $k$ derivatives of a $k$-convex function are in fact monotone \textit{increasing}. All proofs may be easily modified to account for any of these derivatives being monotone decreasing instead. 
% A simple induction argument shows that if $f$ is a $k$-convex function, then $f([N])$ is a $k$-convex set. Thus we may write (\ref{longsumconvexset}) as 
% \begin{equation}\label{longsum[N]}
% |2^kf([N])-(2^k-1)f([N])|\gg_k N^{k+1}.
% \end{equation}
Work by Hanson--Roche-Newton--Rudnev \cite{MishaBrandonOlly} and Bradshaw \cite{peter} in this setting culminates in the following bound.
%peter theorem
\begin{theorem}[Bradshaw \cite{peter}]\label{petersqueeze}
    Let $I$ be an interval and $A \subset I$ be a finite set with $|A+A-A|=K|A|$. Let $k \in \mathbb{N}$ and $f:I \to \R$ be a $k$-convex function. Then we have
    \begin{equation}\label{petersqueezed}
    |2^kf(A)-(2^k-1)f(A)| \gg_k \frac{|A|^{k+1}}{K^{2^{k+1}-k-2}}.
     \end{equation}
\end{theorem}
The sharpness of this inequality can be verified by considering the $k$-convex function $f(x)=x^{k+1}$ and setting $A=[N]$. This also demonstrates the sharpness of Theorems \ref{basicsqueeze} and  \ref{longsumconvexset}.
One should note that $k$-convexity of $f$ is necessary in these inequalities, for if one takes $A=[N]$ and $f(x)=x^k$--an example of a $(k-1)$-convex, but not $k$-convex function--then
$$2^kf([N])-(2^k-1)f([N]) \subset \{-(2^k-1)N^k,\dots,2^kN^k\}$$
and hence
$$|2^kf([N])-(2^k-1)f([N])| \ll_k N^k.$$
So Theorem \ref{petersqueeze} can not hold in this case. We note that a version of this Theorem \ref{petersqueeze} was proved originally by Hanson--Roche-Newton--Rudnev \cite{MishaBrandonOlly} with additional log factors, the removal of which comes from an improvement to the base case
\begin{equation}\label{basicinequality}
    |A+A-A||f(A)+f(A)-f(A)| \gg |A|^3,
\end{equation}
proved via an `equidistribution' argument, see \S\ref{equidistsec}. 
% The $k=1$ case of Theorem \ref{petersqueeze}
% can be compared with the inequality 
% from Theorem \ref{ENR} to demonstrate how squeezing methods are capable of beating state-of-the-art incidence geometric techniques at proving growth of sum sets of convex functions--albeit at the cost of taking longer sums.
\subsection{Linear independence and analyticity}
Note that, by setting $A=g(B)$, one may modify the inequality (\ref{basicinequality}) to prove that
\begin{equation}\label{basicinequality2}
    |f(B)+f(B)-f(B)||g(B)+g(B)-g(B)|\gg |B|^3
\end{equation}
for any two strictly monotone continuously differentiable functions $f,g$ on the same interval $I$ containing $g^{-1}(B)$, such that $f\circ g^{-1}:g(I)\to \R$ is strictly convex or strictly concave. We will use the idea of replacing $A$ with $g(B)$ repeatedly to introduce more functions into our inequalities. Applying the chain rule, one may see that convexity of $f\circ g^{-1}$ is equivalent to the condition that the ratio of derivatives $f'/g'$ is strictly monotone, as long as $g'$ never vanishes. 

It is at this point that we restrict our attention to analytic functions restricted to compact intervals. Analyticity guarantees that for linearly independent and non-vanishing $f'$ and $g'$, $(f'/g')'$ can only vanish at $O_{f,g,I}(1)$ points on $I$. Hence, upon removing these isolated points, $I$ is partitioned into $O_{f,g,I}(1)$ subintervals upon which $f'/g'$ is strictly monotone. It then follows from the pigeonhole principle that one such interval must contain $\Omega_{f,g,I}(|A|)$ points of any finite $A \subset I$. Note that while our implicit constants depend on the functions $f$ and $g$ and the interval $I$, for our applications, we will consider fixed predefined functions over a fixed predefined interval, whence our constants will be genuine. 

In summary, for finite subsets of compact intervals, restriction to analytic functions allows us to use linear independence to detect growth (at the cost of implicit constants, depending on the functions). Therefore, we make the following definition.
\begin{definition}\label{mutcon}
    We say that any set $\mathcal{F}=\{f_0,\dots,f_n\}$ of analytic functions is 1-independent if its set of derivatives $\{f_0',\dots,f_n'\}$ are linearly independent. More generally, we say that $\mathcal{F}$ is $k$-independent, for $k \geq 0$, if its $k$th derivatives $\{f_0^{(k)},\dots,f_n^{(k)}\}$ are linearly independent.
\end{definition}
\textit{Remark:} Since linear dependence of a set of functions implies linear dependence of their derivatives, for any $k \geq 1$, every $k$-independent set of functions is also $(k-1)$-independent. We also remark that at most one function in a set of 1-independent functions can be linear, since otherwise the derivatives would be linearly dependent.

Our focus on analytic functions restricted to a compact interval $I$ allows us to make use of the following properties. For $f$ and $g$ analytic over some open interval containing $I$, we have
\begin{enumerate}
    \item  $f\pm g$, $f\cdot g$, $f \circ g$, $f'$ and $g'$ are analytic.
    \item If $g$ and $g'$ never vanish on $I$, then $1/g$ and $g^{-1}$ are analytic. 
    \item If $f(x)=0$ for infinitely many $x \in I$, then $f\equiv0$.
\end{enumerate}
As we care only for the above behaviour of our functions over a compact interval $I$, we henceforth make the assumption that all functions said to be analytic are analytic on an open interval containing $I$. 
We will also make use of the Wro\'{n}skian to test the linear independence of functions. The Wro\'{n}skian of a set of functions $\{f_0,\dots,f_n:I \to \R\}$ is the function $W(f_0,\dots,f_n):I \to \R$ given by
$$W(f_0,f_1,\dots,f_n)=\det \begin{pmatrix}
    f_0 & f_1 &\dots &f_n\\
    f_0' &f_1' & \dots &f_n'\\
    \vdots &\vdots&\ddots&\vdots\\
    f_0^{(n-1)} &f_1^{(n-1)}&\dots & f_n^{(n-1)}
\end{pmatrix}.$$
For analytic functions, linear independence of the set $\{f_0,\dots,f_n\}$ is equivalent to the Wro\'{n}skian $W(f_0,\dots,f_n)$ not vanishing identically.

\begin{example}
We have the following examples of $k$-independent functions on any compact interval.
\begin{itemize}
\item Any $1$-convex function and the identity function are $1$-independent.
\item Any set of polynomials of different degrees at least $k+1$ are $k$-independent.
\item Any non-trivial finite subset of the trigonometric system $\{\sin (nx), \cos (nx):n \in \mathbb{N}\}$ are $k$-independent for any $k \in \mathbb{N}$.
\end{itemize}  
\end{example}
% for any $d \in \R$ where $\Delta_df(x)\coloneqq f(x+d)-f(x)$ is the discrete $d$-derivative of $f$. Note that, for $f:I \to \R$, the domain of the function $\Delta_df$ is $I \cap (I-\{d\}).$ Here, $d$ will typically be small enough to not have to worry about the size of this domain. 
Despite the sharpness in terms of the power of $|A|$, the dependence on $K$ in Theorem \ref{petersqueeze} is far from optimal. We may use $k$-independence to improve upon this dependence when considering two functions. \begin{theorem}\label{betterKdependence} Let $A$ be a finite subset of a compact interval $I$, with $|A+A-A|=K|A|$ and let $f,g$ be $k$-independent analytic functions. Then
    $$|2^{k}f(A)-(2^{k}-1)f(A)|\cdot|2^{k}g(A)-(2^{k}-1)g(A)| \gg_{f,g,k,I} \frac{|A|^{5\cdot 2^{k-1}-2}}{|A+A-A|^{5\cdot 2^{k-1}-2k-3}} = \frac{|A|^{2k+1}}{K^{5\cdot 2^{k-1}-2k-3}}.$$
\end{theorem}
% In \cite{peter} it is shown that any $k$-convex function $f$ must satisfy 
% \begin{equation}\label{petergrowth}
%     |2^{k}f(A)-(2^{k}-1)f(A)| \gg \frac{|A|^{k+1}}{K^{2^{k+1}-k-2}}
% \end{equation}

%removes log factor from similar result by Rudnev, Hanson, Roche-Newton
Our result says that between a pair of $k$-independent functions, at least one must provide---at the expense of a square root loss from the power of $|A|$---growth with a slightly better dependence on $K$. That is,
$$\max\{|2^{k}f(A)-(2^{k}-1)f(A)|,|2^{k}g(A)-(2^{k}-1)g(A)|\} \gg_{f,g,k,I} \frac{|A|^{k+1/2}}{K^{5\cdot 2^{k-2}-k-3/2}}.$$

This is an improvement over (\ref{petersqueezed}) for sets $A$ with additive tripling $K>|A|^{\frac{1}{3\cdot 2^{k-1}-1}}$.
% or, with $K$ fixed, it is an improvement for $k>\log(1+\frac{\log |A|}{\log K})-\log 3 +1$. 
% For any $\sigma>0$, this yields growth of $\Omega(|A|^\sigma)$ for sets $|A|$ with tripling $K=\Omega(|A|^\frac{1}{2\sigma-1})$ when $k=2\sigma-1$. 
In general, improving the dependence on $K$ allows us to prove growth for sets with less structure. The improved dependence on $K$ here comes from introducing $K$ one step later in the induction, which we are able to do by considering one extra function. 
 Considering enough functions enables us to remove all dependence on $K$ from our inequalities. 
\begin{theorem}\label{growthforsomef}
  Let $\mathcal{F}$ be a set of $n$ 1-independent analytic functions and let $A$ be a finite subset of a compact interval $I$. Then there is a function $f \in \mathcal F$ that satisfies the bound
    $$|2^{n-1}f(A)-(2^{n-1}-1)f(A)|\gg_{\mathcal{F},I} |A|^{\phi(n)},$$
    where $\phi:\mathbb{N} \to \mathbb{R}$ satisfies the recursive formula
    $$\phi(1)=1, \quad \phi(n)=1+\frac{1}{1+\frac{1}{\phi(n-1)}} \quad \text{for } n\geq 2.$$
    % Given any finite set of real numbers $A$, for all $\epsilon>0$ there exists an $n=n(\epsilon) \in \mathbb{N}$ such that any set of $n+1$ 1-independent analytic functions $\mathcal{F}$, contains a function $f \in \mathcal{F}$ satisfying $$|2^sf(A)-(2^s-1)f(A)| \gg_{\mathcal{F}} |A|^{\phi-\epsilon}$$ for some $s \leq n.$
\end{theorem}
\textit{Remark.} Noting that the sequence $\phi(n)$ converges to $\phi:=\frac{1+\sqrt{5}}{2}$ from below, we are unable to prove growth beyond $|A|^\phi$, but are able to get arbitrarily close by considering enough functions and long enough sums.
% Which comes as a consequence of the following (put in proof section).

% Note that $\phi(k)$ converges to the golden ratio $\frac{1+\sqrt{5}}{2}$ from below. Thus taking limits, we may prove the following:

% precisely,

% That is, no matter what set is being considered, there must be a convex function that provides a modest amount of growth under repeated sum set addition.

% There has been work on proving growth under repeated sum set addition for $f(A-A)$ with $f$ a convex function. 
Hanson--Roche-Newton--Senger \cite{HansonRNSenger} show that for any convex function $f$ and finite set of real numbers $A$,
$$|f(A\pm A)+f(A\pm A)-f(A\pm A)| \gg |A|^2.$$
% along the ideas that $(A-A)$ should grow under multiplication. 
One would expect that this bound should be generalisable to longer sums in the same sense that Theorem \ref{basicsqueeze} is generalised to Theorem \ref{longsumconvexset}. Indeed, Bradshaw \cite{peter} conjectures that for any $k$-convex function $f$,
$$|2^kf(A-A)-(2^k-1)f(A-A)| \gg |A|^{k+1}$$
should hold for all finite subsets $A \subset \R$ and $k \in \mathbb{N}$. Previously, the most growth found for iterated sum sets of $f(A-A)$, with any number of summands and for a general class of functions, comes from the following bound due to Roche-Newton \cite[Theorem 3]{Olly}, derived from \cite[Theorem 2.6]{HansonRNSenger}. For any finite set $A \subset \R$, 
\begin{equation}\label{Olly5/2}
    |4f(A-A)-3f(A-A)|\gg|A|^{5/2}
\end{equation}
for any function $f$ satisfying a condition equivalent to $1$-independence of its discrete derivatives. For more specific functions, a lower bound of $|A|^{31/12}$ has been achieved for $7$ sums and differences of $f(A+A)$ by Roche-Newton-Wong \cite{roche2024convexity}, applying recent progress on Elekes-Szab\'{o} estimates \cite{elekes2012find,10.1215/00127094-3674103,solymosi2024improved}. Our methods allow us to generalise \eqref{Olly5/2} to demonstrate for the first time growth beyond $|A|^{5/2}$ for iterated sum sets of $f(A-A)$ for a general class of functions. 
\begin{theorem}\label{f(A-A)} Let $I \subset \R$ be a compact interval, $A \subset I$ be a finite subset and $n \in \mathbb{N}$. Suppose that $f$ is an analytic function such that for any collection of $n$ distinct positive real numbers $0<\delta_1,\dots, \delta_n < \diam (I)$, the functions 
$$\Delta_{\delta_i}f(x)\coloneqq f(x+\delta_i)-f(x), \quad 1 \leq i \leq n$$ are 1-independent, then
$$|2^nf(A-A)-(2^{n}-1)f(A-A)|\gg_{f,n,I} |A|^{
1+\phi(n)}.$$
% where $\phi:\mathbb{N} \to \R$ is the sequence satisfying the recursive formula
% \begin{align*}
% \phi(1)&=1,\\
%     \phi(k)&=1+\frac{1}{1+\frac{1}{\phi(k-1)}} \qquad \text{for all } k\geq 2.
% \end{align*}
\end{theorem}
% Noting that $\lim_{k\to\infty}\phi(k) = \phi \coloneqq \frac{1+\sqrt{5}}{2}$, we have the following.
% \begin{corollary} Let $I$ be a finite interval. For all constant $\epsilon>0$ and sufficiently large $A \subset I$ there exists a $k=k(\epsilon)$ such that for any function $f$ satisfying the above property, must also satisfy 
% $$|2^sf(A-A)-(2^{s}-1)f(A-A)|\gg_{f,\epsilon} |A|^{1+\phi-\epsilon}$$
% for some $s\leq k$.
% \end{corollary}
For any $n\geq 1$, such a function does exist. For instance, we will prove that the property in Theorem \ref{f(A-A)} holds for any polynomial of degree at least $n+1$.
\begin{corollary}\label{polynomials}
   Let $I$ be a compact interval and $n \in \mathbb{N}$. Let $A \subset I$ be a finite non-empty subset and $f$ be any polynomial of degree $m\geq n+1$. Then we have  
   $$|2^nf(A-A)-(2^{n}-1)f(A-A)|\gg_n |A|^{1+\phi(n)}.$$
\end{corollary}
\textit{Remark:} Our methods can be trivially adapted to prove the same results with $A-A$ replaced by $A-B$ (and, in the following subsection, with $A\times A$ replaced by $A \times B$), as long as $|B|=\Theta(|A|)$. We have chosen to focus only on the symmetric case to simplify the exposition.
\subsection{Additive growth for the set of angles in a Cartesian product}
It is an easy exercise in discrete geometry to show, using Beck's Theorem \cite{beck1983lattice}, that the set of distinct angles $\mathcal{A}(P)$ formed by triples of points from a finite non-collinear\footnote{non-collinearity is necessary to avoid the trivial answer of 1} point set $P$ in the plane $\R^2$ satisfies the sharp lower bound 
\begin{equation}\label{beckalicious}
    |\mathcal{A}(P)|\gg|P|.
\end{equation}
 The structural variant of this problem, asking about the qualitative structure of such sets, is much harder. We are aware of essentially only two examples proving the sharpness of \eqref{beckalicious}, both with angles in arithmetic progression. First, consider $2$ points $p,q$ in the plane, together with $n-2$ points arranged on the perpendicular bisector of the line segment between them, forming an arithmetic progression in directions from $p$ and $q$ to points on the bisector. Second, consider the vertices of a regular $(n-1)$-gon, together with its centre. See Figure \ref{fig1}.
\begin{figure}[h]
    \centering
    \begin{minipage}{0.49\textwidth}
    \centering
        \begin{tikzpicture}[scale=0.6,rotate=90,transform shape]
    % Define points A and B
    \coordinate (A) at (-2,0);
    \coordinate (B) at (2,0);
    
    % Draw points A and B
    \fill (A) circle (4pt) node[below] {};
    \fill (B) circle (4pt) node[below] {};
    
    % Compute midpoint M and perpendicular bisector
    \coordinate (M) at ($ (A)!0.5!(B) $);
    \draw[thick] (M) -- ++(0,8) coordinate (T) -- ++(0,-16) coordinate (B1); % Solid perpendicular bisector

    % Define number of points and initial angle
    \def\n{5}  % Number of points above (5 above, 5 below)
    \def\d{4}   % Distance between A and B

    \foreach \i in {0,...,\n} {
        \pgfmathsetmacro{\theta}{15*\i} % Compute angle in degrees
        \pgfmathsetmacro{\h}{(\d/2) * tan(\theta)} % Compute height using trigonometry

        \coordinate (P\i) at (M |- 0,\h);  % Point above
        \coordinate (Q\i) at (M |- 0,-\h); % Symmetric point below
        
        % Draw the points
        \fill (P\i) circle (4pt);
        \fill (Q\i) circle (4pt);

        % Draw auxiliary lines to A and B
        \draw[dotted] (P\i) -- (A);
        \draw[dotted] (P\i) -- (B);
        \draw[dotted] (Q\i) -- (A);
        \draw[dotted] (Q\i) -- (B);
    }
\end{tikzpicture}

    \end{minipage}
    \hfill
    \begin{minipage}{0.49\textwidth}
    \centering
\begin{tikzpicture}[scale=0.6]
    % Define radius
    \def\r{3}  

    % Draw the circle
    \draw[thick] (0,0) circle (\r);

    % Draw the center point
    \fill (0,0) circle (4pt);

    % Draw 12 equally spaced points and connect them to the center
    \foreach \i in {0,30,...,330} {
        % Compute point on the circumference
        \coordinate (P) at ({\r*cos(\i)}, {\r*sin(\i)});
        
        % Draw the point
        \fill (P) circle (4pt);
        
        % Draw dotted lines to center
        \draw[dotted] (0,0) -- (P);
    }

\end{tikzpicture}
    \end{minipage}
\caption{Point sets $P$ determining $\Omega(|P|)$ distinct angles}
\label{fig1}
\end{figure}
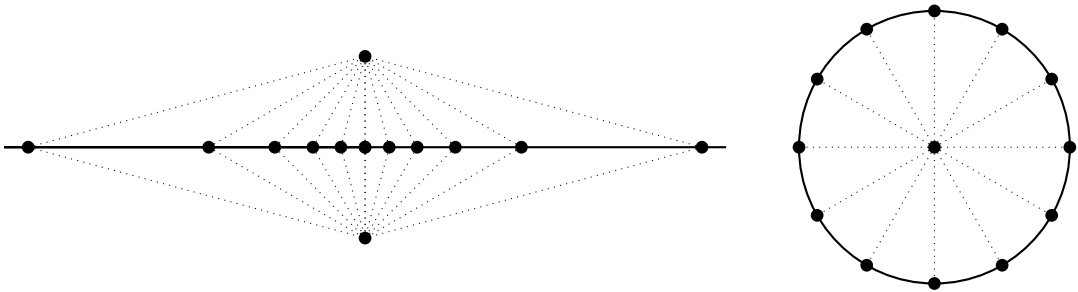

Konyagin-Rudnev-Passant \cite{KPR} conjecture that all such examples come, essentially, from one of these sets. Towards this structural problem, it is pertinent to prove, for other types of sets, that $|\mathcal{A}(P)|\gg |P|^{1+c}$ for some absolute constant $c$. Recent work by Konyagin, Passant and Rudnev \cite{KPR} proves that for non-co-circular point sets $P$ in convex position satisfy the bound $$|\mathcal{A}(P)| \gg_\epsilon |P|^{1+3/23-\epsilon}$$
and even more recently, Roche-Newton \cite{Olly} proved that any Cartesian product $A\times A \subset \R^2$ must satisfy
\begin{equation}\label{ollybound}
    |\mathcal{A}_p(A\times A)| \gg |A|^{2+{1/14}}
\end{equation}
where $\mathcal{A}_p(P)$ denotes the set of \textit{pinned} angles at $p$ i.e. the set of angles formed between $p$ with any other two points in $A\times A$. This is done using Theorem \ref{Olly5/2} for the function $f(x)=\arctan(e^x)$ to show that $\mathcal A(A\times A)$ exhibits additive growth upon which the Pl\"{u}nnecke-Ruzsa inequality yields the stated bound. We check that the condition given in Theorem \ref{f(A-A)} for the function $f(x)=\arctan(e^x)$ with $n=3$, and substitute the resulting bound into the first part of Roche-Newton's proof for an improved bound on the additive growth of $\mathcal{A}(A \times A)$.
\begin{corollary}\label{introangles} Let $A\subset\R$ be a finite set. Then for any point $p \in A\times A$ we have
    $$|8\mathcal{A}_p(A\times A)|\gg |A|^{13/5}=|A|^{2.6}.$$ 
\end{corollary}
 Unfortunately, this bound does not allow us to improve upon \eqref{ollybound} since the maximum length of our sums for $n>2$ renders the Pl\"{u}nnecke-Ruzsa inequality counterproductive.

\subsection{Notation}
We will make extensive use of Vinogradov's symbols $\ll$ and $\gg$ to suppress constant terms in inequalities: We write $X \ll_{z} Y$, or equivalently $Y \gg_{z} X$, to mean $X \geq C({z}) |Y|$ where $C(z)$ is some positive constant depending on the parameter $z$. We further write $X = O_{z}(Y)$ or $Y=\Omega_z(X)$ to mean $X \ll_{z} Y$ and $X=\Theta(Y)$ if $Y\ll_z X\ll_z Y$.
\subsection{Structure of paper}
In \S \ref{prelims} we familiarise the reader with previous iterations of the squeezing argument due to Ruzsa-Shakan-Solymosi-Szemer\'{e}di \cite{shakanetc} and Hanson--Roche-Newton--Rudnev \cite{MishaBrandonOlly}, before proving various lemmata extending these tools to the realm of 1-independent functions. We then restate the equidistribution arguments of Bradshaw \cite{peter} used to prove \ref{petersqueeze} and reprove a version of \eqref{basicinequality}, which is later used as a base case in many of our arguments. In \S \ref{proofs}, we prove Theorems \ref{betterKdependence}, \ref{growthforsomef} and \ref{f(A-A)} and in \S \ref{corollaries} we prove Corollaries \ref{polynomials} and \ref{introangles}.
\subsection{Acknowledgments}
The author would like to thank Misha Rudnev for many valuable discussions. He would also like to thank David Ellis for helpful comments, Oliver Roche-Newton for pointing to the reference \cite{roche2024convexity} and Oleksiy Klurman and Besfort Shala for pointing out the relevance of Hurwitz's Theorem.  
\section{Preliminaries}\label{prelims}
We begin this section by recapping the squeezing arguments used to prove Theorems \ref{basicsqueeze}, \ref{longsumconvexset} and \ref{petersqueeze}. Readers already familiar with these techniques might wish to skip ahead to \S \ref{mutualconvexprelims}.
% In this section we familiarise the reader with the squeezing arguments of Ruzsa, Shakan, Solymosi and Szemer\'{e}di \cite{shakanetc} and Hanson, Roche-Newton and Rudnev \cite{MishaBrandonOlly}, and the equidistribution arguments of Bradshaw \cite{peter}.
\subsection{Squeezing}
Our proofs are based heavily on the \textit{squeezing} methodology, whose genealogy is traced back to Garaev \cite{garaev2000lower}, via work of Bradshaw \cite{peter}, Hanson--Roche-Newton--Rudnev \cite{MishaBrandonOlly}, and Ruzsa-Shakan-Solymosi-Szemeredi \cite{shakanetc}. The basic idea behind this method is that convexity of a set $A=\{a_1<\dots<a_{|A|}\}$ allows us to construct $i$ elements of $A+A-A$ within the $i$th consecutive interval $(a_i,a_{i+1}]$ by merely noting that every interval to its left is shorter and thus can be `\textit{squeezed} inside'. See the following proof of Theorem \ref{basicsqueeze}, appearing first in \cite{shakanetc}.
\begin{proof}[Proof of Theorem \ref{basicsqueeze}]  
Write $A=\{a_1<...<a_N\}$ and pick some $i \in {1,...,N}$. By convexity of $A$ we have, for all $j\leq i$,    
    $$0 < a_{j+1}-a_{j} \leq a_{i+1}-a_{i}.$$
Rearranging, we get
    $$ a_i < a_i+(a_{j+1}-a_{j}) \leq a_{i+1}. $$
But $a_i+a_{j+1}-a_{j}$ is an element of $A+A-A$, giving us at least $i$ elements of $A+A-A$ in the interval $(a_i,a_{i+1}]$. Hence
$$|A+A-A|=\sum_{i=1}^{N-1}|(a_i,a_{i+1}]\cap(A+A-A)|>\sum_{i=1}^{N-1}i \gg |A|^2$$
as desired.
\end{proof}
The proof of Theorem \ref{longsumconvexset} given in \cite{MishaBrandonOlly} merely requires an iterated application of the above squeezing argument. Previously, for each $i \in \{1,...,|A|\}$, we identified $i$ elements of $A-A$ that can be squeezed into intervals $(a_i,a_{i+1}]$, each producing a unique element of $A+A-A$. 

    Let us consider a second-order iteration of this argument. For a $2-$convex set $A$, denoting the set of consecutive differences by $D \coloneqq \{d_i=a_{i+1}-a_i:1\leq i \leq |A|\}$, we must have the following inequalities
    \begin{align*}
        d_1<d_2<&...<d_{N-1}, \\
        d_2-d_1<d_3-d_2<&...<d_{N-1}-d_{N-2}.
    \end{align*}

We now squeeze elements of $(A-A)-(A-A)$ into the interval $(a_i+d_j,a_i+d_{j+1}]$. For $j \leq i$ fixed, we have for each $k\leq j$ that
$$ a_i+d_j < a_i+d_j+(d_{k+1}-d_k) \leq a_i+d_{j+1},$$
accounting for $\sum_{i=1}^{|A|} \sum_{j \leq i} j = \Omega( |A|^3)$ elements of the set $$A+(A-A)+((A-A)-(A-A))=4A-3A.$$
That is to say, the previously established fact that $$|(2D- D)\cap (a_2-a_1, a_N-a_{N-1})| \gg |D|^2 \approx N^2$$ allows us to find $\Omega(N^2)$ elements from the set $$a_i+2D-D \subseteq a_i+(A-A)+(A-A)-(A-A)=a_i+3A-3A$$ in the interval $(a_i,a_{i+1}]$, accounting for $\Omega(\sum_{i=1}^{N}N^2)=\Omega(N^3)$ elements of $4A-3A$.

It is now clear how one generalises this into an induction argument: the following proof appears in \cite{MishaBrandonOlly}.
\begin{proof}[Proof of Theorem \ref{longsumconvexset}]
The inductive hypothesis tells us that $$|(2^kD - (2^{k}-1)D)\cap (a_2-a_1,a_N-a_{N-1})| \gg N^{k+1},$$ which allows us to find $\Omega(N^{k+1})$ elements from the set $$a_i+2^kD - (2^{k}-1)D \subseteq a_i+(2^{k+1}-1)A - (2^{k+1}-1)A$$ in the interval $(a_i,a_{i+1}]$, accounting for $\Omega(\sum_{i=1}^N N^{k+1})=\Omega(N^{k+2})$ elements of $2^{k+1}A - (2^{k+1}-1)A$. Combined with the base case of Theorem \ref{basicsqueeze}, this proves the theorem. 
\end{proof}

To extend this argument to images of $1$-convex functions, we define the discrete $d$-derivative of $f$ as follows.
\begin{equation}
    \Delta_df(x)\coloneqq f(x+d)-f(x).
\end{equation}

If $f$ is $1$-convex, then for any $d>0$, $\Delta_df(x)$ is an increasing function. Indeed, by the fundamental theorem of calculus,
\begin{equation}\label{FTC}
   \Delta_df(x) = \int_x^{x+d}f'(t)dt 
\end{equation} and since $f'$ is increasing, so too must be $\Delta_df$. Thus, if $A=\{a_1<\dots<a_{|A|}\}$, we must have
$$f(a_i+d)-f(a_i)<f(a_j+d)-f(a_j)$$
for all $j>i$. Hence, the terms
$$f(a_i)+f(a_j+d)-f(a_j)$$ are all distinct elements of $(f(a_i),f(a_i+d)].$

We therefore have the following result, known as `the squeezing lemma' \cite{MishaBrandonOlly,peter}.

\begin{lemma}[The Squeezing Lemma]\label{squeezing lemma}
    Let $f$ be a 1-convex function. Then for any $d>0$, $a \in \R$ and any set $S$ with $\sup S < a$, the set $\Delta_{d,a}f(S)=f(a)+\Delta_df(S)$ is entirely contained in the interval $(f(a),f(a+d)).$
\end{lemma}

\subsection{$k$-Independence lemmata}\label{mutualconvexprelims} 
% The inductive steps in the proofs of Theorems \ref{betterKdependence} and \ref{growthforsomef} will require sets of functions constructed in such a way that may technically not be 1-independent per se, but are indistinguishable from a set of 1-independent functions from the perspective of the finite set $A$. For this we require the following definition.
% \begin{definition}Let $A$ be a finite set and let $k \geq 0$. A set of $(k-1)$-convex functions $\mathcal{F}$ is called $k$-independent on $A$ if for any pair of functions $f,g \in \mathcal{F}$ the functions $f^{(s)}/g^{(s)}$ are strictly monotone on $A$ for all $s \leq k$.
% \end{definition}
% \textit{Remark:} Of course, any set of $k$-independent functions are $k$-independent on any finite set.
To extend squeezing techniques to the realm of $k$-independent functions, we require the following lemmas.

\begin{lemma}\label{fginv} Let $\mathcal{F}$ be any set of 1-independent analytic functions. Then for any ordering $\{f_i\}_{i=1}^n$ of $\mathcal{F}$, the set $\{(f_i\circ f_n^{-1})'\}_{i=1}^n$ is a set of 1-independent analytic functions.  
\end{lemma}
\begin{proof}
%     Again we proceed by induction on $k$. The $k=1$ case follows from the chain rule: For any pair of functions $f,g \in \mathcal{F}$ we have 
%     $$\left(\frac{(f \circ h^{-1})'}{(g\circ h^{-1})'}\right)^{(s)}=\frac{(f'/h')}{(g'/h')}\circ h^{-1}=(f'/g')\circ h^{-1}$$ 
% Which is strictly monotone since both $f'/g'$ and $h$ are strictly monotone. 
% We proceed by induction on $n$. The $n=1$ case is trivial. Now suppose that the statement is true for for all sets of $n-1$ 1-independent analytic functions and consider a set of $n$ 1-independent analytic functions $\{f_i\}_{i=0}^n$. 
Recall from the definition of $1$-independence for analytic functions that the derivatives $\{f_i'\}_{i=1}^n$ are non-vanishing and linearly independent. Since the derivatives are non-vanishing, the functions themselves must be strictly monotone and hence invertible. By the chain rule, we have for any $f,g \in \mathcal{F}$ that
$$(f\circ g^{-1})'=(f'/g') \circ g^{-1}.$$
Which cannot vanish due to the non-vanishing of $f'$.

It just remains to prove prove that $\{(f_i\circ f_n^{-1})''\}_{i=1}^{n-1}$ is a set of linearly independent functions. Noting from the chain rule that we have $$(f\circ g^{-1})''=\frac{W(g',f')}{(g')^3} \circ g^{-1}$$ for any pair of functions $f,g \in \mathcal{F}$, we consider a linear combination of our functions
\begin{align*}
\sum_{i=1}^{n-1}a_i(f_i\circ f_n^{-1})''&=\left(\frac{1}{(f_n')^3}\sum_{i=1}^{n-1} a_iW(f_n',f_i')\right) \circ f_n^{-1}\\
&=\left(\frac{1}{(f_n')^3}\sum_{i=1}^{n-1} a_i(f_i''\cdot f_n'-f_n''\cdot f_i')\right)\circ f_n^{-1}\\
&=\left(\frac{1}{(f_n')^3}\left(\left(\sum_{i=1}^{n-1}a_i f_i''\right)\cdot f_n'-f_n''\cdot\left(\sum_{i=1}^{n-1}a_if_i'\right)\right)\right)\circ f_n^{-1}\\
&=\left(\frac{1}{(f_n')^3}W\left(f_n',\left(\sum_{i=1}^{n-1} a_if_i'\right)\right)\right)\circ f_n^{-1}
\end{align*}
with the $a_i \in \R$ constants. By linear independence of the set of functions $\{f_i'\}_{i=1}^n$, the Wro\'{n}skian in the last line can only vanish everywhere if $a_i=0$ for every $1\leq i \leq n-1$, and by injectivity of $f_n$, this is necessary for the whole expression to vanish. Whence $\{(f_i\circ f_n^{-1})''\}_{i=1}^{n-1}$ is a linearly independent set of functions and $\{(f_i \circ f_n^{-1})'\}_{i=1}^n$ is 1-independent. 
\end{proof}

We will later need the above recursive property to hold with the derivatives $f'(x)$ replaced by discrete derivatives $\Delta_df(x) \coloneqq f(x+d)-f(x)$, with small enough $d$.

To prove that we can do this, we will use the following standard result from complex analysis known as Hurwitz's Theorem. See \cite[Theorem 2.5]{markushevich1965theory}. 
\begin{lemma}[Hurwitz's Theorem]
    Given a closed contour $L \subset \mathbb{C}$ with interior $K$, suppose that $(F_m)_{m=1}^\infty$ is a sequence of analytic functions on $K$ which converges uniformly on $K$ to a function $F\not\equiv 0$ that does not vanish on $L$. Then there is an integer $M=M(L)>0$ such that for all $m>M$ the functions $F_m$ and $F$ have the same number of zeros in $K$.  
\end{lemma}
\begin{corollary}\label{goodN}
Let $\{g_1,\dots,g_n\}$ be a collection of linearly independent real analytic functions defined on a compact interval $I$. There exists an $N=N(g_1,\dots,g_n,I)$ such that any non-trivial linear combination $\sum_{i=1}^n a_if'_i$ has at most $N$ zeros.
\end{corollary}
\begin{proof}
First, let us extend our functions $g_i$ to a connected and simply connected open set $\Omega \subset \mathbb{C}$ such that $I \subset \Omega$. Now suppose the claim is false. Then for any $m>0$ there exists a set of non-zero coefficients $\{a_i^{(m)}\}_{i=1}^n$ such that the function $F_m\coloneqq \sum_{i=1}^n a_i^{(m)}g_i$ has at least $m$ zeros. Dividing by the largest coefficient, we may assume that the $a_i^{(m)}$ all lie in the interval $[-1,1]$. Write these sets of coefficients as a sequence of vectors $\mathbf{a}_m\coloneqq (a_1^{(m)},\dots,a_n^{(m)}) \in [-1,1]^n$. Compactness of $[-1,1]$ guarantees that the sequence of vectors $(\mathbf{a}_m)$ contains a convergent subsequence. Hence, the corresponding subsequence of functions $F_m$ converges uniformly to a complex analytic function $F$. Since the zeros of $F$ are isolated, there must exist a closed contour $L=L(F) \subset \Omega$ with compact interior $K$ containing $I$ such that $F$ does not vanish on $L$. By Hurwitz's Theorem, either $F\equiv 0$ on $K$ or there is some $M=M(L) > 0$ such that for all $m>M$, $F_m$ and $F$ have the same number of zeros in $K$. Both of these possibilities yield contradictions since $F$ is a non-trivial linear combination of linearly independent functions and we have assumed that the number of zeros of the $F_m$ are strictly increasing. 

Since $M$ depends on $L$, which depends ultimately only on the functions $g_i$ and the interval $I$, we have the stated dependence.
\end{proof}
We are now ready to prove the remaining lemmas.

\begin{lemma}\label{f'todeltaf}
    Let $\mathcal{F}=\{f_i\}_{i=1}^n$ be a set of 1-independent analytic functions defined over a compact interval $I$. Then there exists some $N=N(\mathcal
    F,I)$ such that set of discrete derivatives $\{\Delta_df_i\}_{i=1}^n$ is linearly independent for any $0< d < \diam (I)/N$. Further, if $\mathcal
    F$ is $k$-independent, then for $\{\Delta_df_i\}_{i=1}^n$ is $(k-1)$-independent for any $0<d<\diam(I)/N$.
\end{lemma}
\begin{proof}
By Corollary \ref{goodN}, there is an $M=M(\mathcal
F,I)$ such that any non-trivial linear combination of the linearly independent derivatives $f_i'$ has at most $M$ zeros on $I$. Set $N=M+1$ and take $0<d<\diam(I)/N$.
Now, consider a linear combination $\sum_{i=1}^n a_i\Delta_df$ and suppose that it vanishes for all $x \in I \cap (I-d)$.
We have, for all $x$,
        \begin{align}\label{tobecontradicted}
            0=\sum_{i=1}^n a_i\Delta_df_i(x)&=\sum_{i=1}^n a_i\int_x^{x+d}f_i'(t)dt=\int_x^{x+d}\sum_{i=1}^na_if_i'(t)dt.
        \end{align} 
If the integrand $\sum_{i=1}^na_if_i'$ on the right vanishes for all $t \in [x,x+d]$, then $\sum_{i=1}^na_if_i' \equiv 0$ by analyticity. It would then follow by linear independence of $f'_i$ that $a_i=0$ for all $1 \leq i \leq n$, proving linear independence of the $\Delta_df_i$. Suppose, then, that the integrand $\sum_{i=1}^na_if_i'$ does not vanish identically on $[x,x+d]$ for any $x$. We have established that $\sum_{i=1}^na_if_i'$ can have at most $M$ zeros on $I$, hence there exists an open subinterval of length $\diam(I)/(M+1)$ where $\sum_{i=1}^na_if_i'$ never vanishes. Taking $x$ to be the left endpoint and noting that $0<d<\diam(I)/(M+1)$, we have that its integral $\int_x^{x+d}\sum_{i=1}^na_if_i'(t)dt$ cannot vanish, contradicting \eqref{tobecontradicted}. Whence linear independence of the $\Delta_df_i'$ is the only possibility. 

% Suppose that the set $\{f_i'\}_{i=1}^n$ is linearly dependent. That is, there exist non-zero constants $a_1,\dots,a_n \in \R$ such that the linear combination $\sum_{i=1}^na_if'(t)=0$ for all $t \in I$. Now, for any $x \in I \cap (I-d)$, consider the linear combination $\sum_{i=1}^n a_i\Delta_df_i(x)$. We have
%         \begin{align*}
%             \sum_{i=1}^n a_i\Delta_df_i(x)&=\sum_{i=1}^n a_i\int_x^{x+d}f_i'(t)dt=\int_x^{x+d}\sum_{i=1}^na_if_i'(t)dt=0.
%         \end{align*} 
%     Hence the functions $\{\Delta_df_i\}_{i=1}^n$ are linearly dependent.     
    % Since the above expression is equal to $0$ for all $x$, and writing $b'=a+\floor{\frac{b-a}{d}}\cdot d$  
    %  \begin{align*}
    %      \int_x^{b'} \sum_{i=1}^n a_if_i'(t)dt=\sum_{m=1}^{(b'-a)/d} \int_{x+(m-1)d}^{x+md}\sum_{i=1}^n a_if_i'(t)dt=0
    %  \end{align*}
    % for all $x \in \R$. Whence the integral must vanish on all intervals $[x,y]$. Indeed, for any $x\neq y$ we have
    % $$\int_x^y\sum_{i=1}^n a_if_i'(t)dt=\int_y^\infty\sum_{i=1}^n a_if_i'(t)dt-\int_x^\infty\sum_{i=1}^n a_if_i'(t)dt=0.$$
    % For any integral vanishing on every interval, its integrand must vanish almost everywhere. Since our integrand is an analytic function, this is only possible if $\sum_{i=1}^n a_if_i'(t)=0$ for all $t$. Which, by linear independence of the $f_i'$, is only possible if $a_i=0$ for all $i=1,\dots,n$. 
    To see why the second statement is true, note that for any $f \in \mathcal{F}$, $(\Delta_df)^{(k-1)}=\Delta_d(f^{(k-1)})$. Whence linear independence of the $(\Delta_df_i)^{(k-1)}$ follows from linear independence of the $f_i^{(k)}$. Non-vanishing of the $\Delta_df^{(s-1)}$ for each $1 \leq s \leq k$ comes from the non-vanishing of $f^{(s)}$ as follows: We have, for all $x \in I \cap (I-d)$
    $$\Delta_d f^{(s-1)}(x)=\int_x^{x+d}f^{(s)}(t)dt.$$
    And the integral on the right cannot vanish as its integrand is continuous and non-vanishing on the range of integration.
\end{proof}
Combining the previous two lemmas, we have the following.
\begin{lemma}\label{f'todeltaf2}
    Let $n\geq 1$ and let $\mathcal{F}$ be any set of $n$ 1-independent analytic functions defined over an interval $I$. Then exists an $N=N(\mathcal{F},I)$ such that for any $0<d<\diam(I)/N$ and any ordering $\{f_i\}_{i=1}^n$ of $\mathcal{F}$ the set $\{\Delta_d(f_i\circ f_n^{-1})\}_{i=1}^{n-1}$ is a set of $n$ 1-independent analytic functions.
\end{lemma}
\begin{proof}
By Lemma \ref{fginv}, we know that $\{(f_i\circ f_n^{-1})'\}_{i=1}^{n-1}$ is a set of 1-independent analytic functions, i.e. the set $\{(f_i\circ f_n^{-1})''\}_{i=1}^{n-1}$ is linearly independent. By Lemma \ref{f'todeltaf}, the set $\{\Delta_d(f_i\circ f_n^{-1})'\}_{i=1}^{n-1}$ is linearly independent, and hence the set $\{\Delta_d(f_i \circ f_n^{-1})\}_{i=1}^{n-1}$ is a set of 1-independent analytic functions.
\end{proof}
% \begin{lemma}\label{f'todeltaf2} Let $A$ be a finite set of real numbers and let $d$ be its smallest consecutive difference. Let $\mathcal{F}=\{f_i\}_{i=0}^{k}$ be a set of $k+1$ mutually $2$-convex functions for $k\geq 1$. Then $\{\Delta_d(f_i\circ f_k^{-1})\}_{i=0}^{k-1}$ is a set of 1-independent functions on $A$ for all $d \in D$. 
% \end{lemma}
% \begin{proof}
% By the previous lemma, we know that the set $\left\{\left(f_i\circ f_k^{-1}\right)'\right\}_{i=0}^{k-1}$ is mutually $(k-1)$ convex and so by Lemma \ref{f'todeltaf}, $$\left\{\left(\Delta_d\left(f_i\circ f_k^{-1}\right)\right)'\right\}_{i=0}^{k-1}=\left\{\Delta_d\left(\left(f_i\circ f_k^{-1}\right)\right)'\right\}_{i=0}^{k-1}$$ is mutually $(k-2)$-convex on $A$. Whence $\{\Delta_d(f_i\circ f_k^{-1})\}_{i=0}^{k-1}$ is $(k-1)$-independent on $A$.  % Again we proceed by induction on $k$. The base case is the same as in Lemma \ref{f'todeltaf}. Suppose that the statement holds for all sets $\mathcal{G}$ of $(k-1)$-convex functions with $|\mathcal{G}|=k$ and consider any set $\mathcal{F}=\{f_i\}_{i=0}^{k}$ of $k$-convex functions. 
% \end{proof}

\subsection{Equidistribution}\label{equidistsec}

We introduce the following notation from \cite{peter}: If $a'<a$, then
$$n_A(a,a') \coloneqq |(2A-A)\cap(a',a]|.$$
Looking over each pair $a,a' \in A$, the quantity $n_A(a,a')$ describes how the elements of $2A-A$ are distributed amongst the gaps between elements of $A$. 

The following lemma, first proved in \cite{peterenergy}, says that, in some sense, the set $A+A-A$ is evenly distributed between these gaps. See \cite{peter, peterthesis} for more details. 

\begin{lemma}\label{equidistribution} let $D:=\{d_1<\dots<d_{|D|}\}$ be the set of positive differences in $A-A$. Suppose $a, a' \in A$ with $a'<a$ and $n_A(a',a) \leq N$. Then $a-a' \leq d_N$.
\end{lemma}
\begin{proof}
Suppose not, i.e. that there exists some $M > N$ for which $a-a'=d_M$. Then we have for each $i \leq M$,
$$a' < a'+d_i \leq a$$

accounting for at least $M>N$ elements of $A+A-A$ in $(a',a]$, contradicting the assumption that $n_A(a',a) \leq N$. 
% CHANGE TO CONTRAPOSITIVE
\end{proof}

We now restate and prove a version of (\ref{basicinequality2}). The proof here is essentially identical to Bradshaw's \cite{peter}, but we include it for completeness as it will serve as a base case for our later inductions. 

\begin{lemma}\label{basecase} 
Let $I$ be a compact interval, and let $A \subset I$ be a finite subset. Let $f$ and $g$ be 1-independent analytic functions. Then % such that the map $\Psi_{f,g}:(x,y) \mapsto (f(x)-f(y),g(x)-g(y))$ is injective, then
$$n_{f(A)}(\min f(A),\max f(A)) \cdot n_{g(A)}(\min g(A),\max g(A)) \gg_{f,g,I} |A|^3.$$
    % $$\mathcal{N}_1(f(A)) \cdot \mathcal{N}_1(g(A)) \gg |A|^3$$
\end{lemma}
\begin{proof} Noting the discussion preceding Definition \ref{mutcon}, at the cost of introducing implicit constants dependent on the functions $f$ and $g$, we may work on a subinterval where $f\circ g^{-1}$ is strictly convex or strictly concave. For simplicity we assume that $f \circ g^{-1}$ is $1$-convex (i.e. strictly convex and strictly increasing), noting that the other cases follow mutatis mutandis.   

Let $F(x):=f\circ g^{-1}(x)$ and set $B=g(A)$. Writing $B=\{b_1<...<b_{|A|}\}$, we say that an element $b_i \in B$ is \textit{good} if
\begin{equation}\label{good}
    n_{B}(b_i,b_{i+1})  \leq \frac{4n_{B}(b_1,b_{|B|})}{|B|}  \ \text{  and  } \ n_{F(B)}(F(b_i),F(b_{i+1}))  \leq \frac{4n_{F(B)}(F(b_1),F(b_{|B|}))}{|B|}. 
\end{equation}
% that is, if the number elements of $2f(A)-f(A)$ in the interval $(f(a_i),f(a_{i+1}))$ is no more than $C_1$ times the average number of elements of $2f(A)-f(A)$ that lie between the extremes of $A$, and the equivalent for $2g(A)-g(A)$. If fewer than $|A|/C$ elements of $A$ are good, then for both the $f$ and $g$ cases above more than $|A|/C$ consecutive intervals have, respectively,  $CX_f$ and $CX_g$ elements, producing at least $C|A|X_f=CT>T$ elements -- contradiction
By the pigeonhole principle, one can find a subset $B' \subseteq B$ with $|B'| \geq |B|/2$ consisting only of good $b_i$. 
% Now consider the map $\Psi:(x,y) \mapsto (x-y,F(x)-F(y))$.
By the mean value theorem, we have, for all consecutive pairs $(b_i,b_{i+1})$ in $A$, that there exists a $c_i \in (b_i,b_{i+1})$ such that
$$\frac{F(b_{i+1})-F(b_i)}{b_{i+1}-b_i}=F'(c_i).$$
From the convexity of $F$, we know that $F'$ is strictly monotone. Thus knowing $F(b_{i+1})-F(b_i)$ and $b_{i+1}-b_{i}$ uniquely determines $c_i$. In turn, by disjointness of the intervals $(b_i,b_{i+1})$, knowing $c_i$ uniquely determines the $b_i$. It follows that the map $\Psi:b_i \mapsto(b_{i+1}-b_i,F(b_{i+1})-F(b_i))$ is injective. Whence
% $$\frac{F(b_{i+1})-F(b_i)}{b_{i+1}-b_i}=F'(c).$$
% Since $F$ is convex, $F'$ is monotone and so knowing $F(b_{i+1})-F(b_i)$ and $b_{i+1}-b_i$ uniquely determines $c$ and since the intervals $(b_i,b_{i+1})$ are all disjoint, $c$ uniquely determines $b_i.$ The result follows by noting that monotonicity of $g$ preserves the ordering of $A$ and hence $b_i=g(a_i)$ for all $i=1,\dots,|A|.$
% Since $f$ and $g$ are 1-independent, their derivatives $f'$ and $g'$ are linearly independent. Thus the Jacobian of $\Psi_{f,g}$, $$J\Psi_{f,g}(x,y)\coloneqq \begin{pmatrix}
% f'(x) & g'(x)\\
% f'(y) & g'(y) 
% \end{pmatrix}$$
% Has, for all $x \neq y$, linearly independent columns and hence its determinant does not vanish. Thus $\Psi_{f,g}$ is locally invertible at every point $(x,y)$ off the diagonal.
% In particular it is locally invertible for all consecutive pairs in $A$ with the smallest element from $A'$,
\begin{equation}\label{Psi}
    |B| \ll |B'|-1 = |\Psi(B'\setminus \{b_{|B|}\})|.
\end{equation}
Since $B'$ consists only of good elements, we have that (\ref{good}) holds for each $b_i \in B'$. Lemma \ref{equidistribution} then shows that for any $b_i \in B'$, $b_{i+1}-b_i$ and $F(b_{i+1})-F(b_i)$ are respectively amongst the smallest $4n_{B}(b_1,b_{|B|})/|B|$ and $4n_{F(B)}(F(b_1),F(b_{|B|}))/|B|$ elements of $B-B$ and $F(B)-F(B)$. Hence
$$|\Psi(\{(b_i,b_{i+1}): b_i \in B'\})| \ll \frac{n_{B}(b_1,b_{|B|})n_{F(B)}(F(b_1),F(b_{|B|}))}{|B|^2},$$
which, in combination with (\ref{Psi}) and the monotonicity of $F$, proves that
\begin{align*}
|B|^3 \ll n_{B}(\min B,\max B) &\cdot n_{F(B)}(\min F(B),\max F(B)).
\end{align*}
The result follows upon replacing $B$ with $g(A)$.
\end{proof}

From this, by taking $g$ as the identity function and $f$ as any convex function, we recover (\ref{basicinequality}) proved in \cite{peter}. 

% \textit{Remark:} Observe that if $f(x)-g(x)$ is (piecewise) monotone, then in the above lemma one can assume that the image pair has, but for a finite number of $a_i$, distinct entries 
% $$f(a_{i+1})-f(a_i) \neq g(a_{i+1})-g(a_i).$$
% The basis of our extension to multiple functions begins with the observation that if $f$ and $g$ are 1-independent, then one may replace in the above theorem $A$ with $g(A)$ and $f$ with $f \circ g^{-1}$ to prove the following

% \begin{corollary}\label{basecase} Let $A$ be a set of real numbers and $f,g: \R \to \R$ be strictly monotone and 1-independent functions then 
%     $$|f(A)+f(A)-f(A)||g(A)+g(A)-g(A)| \gg |A|^3$$
% \end{corollary}

\section{Proofs of Theorems \ref{betterKdependence}, \ref{growthforsomef} and \ref{f(A-A)}}\label{proofs}

  To prove Theorems \ref{betterKdependence} and \ref{growthforsomef}, we will, in fact, prove slightly stronger statements where the same bounds hold for the appropriate set(s) $B$ and value(s) $k$, each iterated sum set $2^kB-(2^k-1)B$ lies within the interval spanned by $B$. To this end, we introduce the following notation: For any set $B \subseteq \R$ and $s \in \mathbb{N},$
$$\mathcal{N}_k(B) \coloneqq |(2^kB-(2^k-1)B)\cap (\min B, \max B)|.$$
The statements given previously follow from the simple observation that, for any set $B$
$$|2^kB-(2^k-1)B| \geq \mathcal{N}_k(B).$$
In this notation, Lemma \ref{basecase} says that 1-independent analytic functions $f,g$ on a compact interval $I$ satisfy, for all finite subsets $A \subset I$, the following inequality.
\begin{equation*}
   \mathcal{N}_1(f(A))\cdot\mathcal{N}_1(g(A)) \gg_{f,g,I} |A|^3.
\end{equation*}
The following technical result concerning the quantity $\mathcal{N}_k(B)$, a form of the squeezing lemma, will be important for our proofs.
\begin{lemma}\label{B_d} Let $J$ be an interval and $k \geq 2$. Suppose that $g \in C^1(J)$ is a $1$-convex function and let $B=\{b_1 < \dots < b_{|B|}\}$ be a finite subset of $J$. Let $D$ be the set of consecutive differences $b_{i+1}-b_i$ of $B$. For each $d \in D$, we define the set $B_d \coloneqq \{b_i : b_{i+1}-b_i=d\}$ and let $B_d(i)=\{b_{e_1},\dots,b_{e_i}\}$ be the set of the smallest $i$ elements of $B_d$ for each $i=1, \dots, |B_d|.$ For each $d \in D$ and $1 \leq i \leq |B_d|$, define the function $\Delta_{d,b_{e_i}}g$ by 
 $$\Delta_{d,b_{e_i}}g(x) \coloneqq g(b_{e_i})+\Delta_dg(x).$$
Then,
    \begin{equation}\label{squeezing2}
    \mathcal{N}_k(g(B))\gg \sum_{d \in D} \sum_{i=1}^{|B_d|} \mathcal{N}_{k-1}(\Delta_{d,b_{e_i}}g(B_d(i))).
\end{equation}
\end{lemma}

In \cite{peter}, this inequality is combined with a simple inductive argument using (\ref{basicinequality}) as a base case to prove Theorem \ref{petersqueeze}. Our proofs follow a similar structure with Lemma \ref{basecase} as our base case and an augmented inductive step.

\begin{proof}
% Lemma \ref{FFC}, the function $\Delta_{d,b_{e_i}}g(x)=g(b_{e_i})+\Delta_df(x)$ is a $(k-1)$-convex function and by 

We have, by definition, that $B_d(i)$ is a subset of $I \cap(I-d)$, the domain of $\Delta_dg$. Moreover, by Lemma \ref{squeezing lemma} we have that $\Delta_{d,b_{e_i}}g(B_d(i)) \subset (g(b_{e_i}),g(b_{e_{i+1}})]$ and note that for each $d,i$
$$2^{k-1}\Delta_{d,b_{e_i}}g(B_d(i))-(2^{k-1}-1)\Delta_{d,b_{e_i}}g(B_d(i)) \subset 2^kg(B)-(2^k-1)g(B).$$

Therefore, over each $d \in D$ and then $1 \leq i \leq |B_d|$, the sets $$(2^{k-1}\Delta_{d,b_{e_i}}g(B_d(i))-(2^{k-1}-1)\Delta_{d,b_{e_i}}g(B_d(i)))\cap (g(b_{e_i}),g(b_{e_{i+1}})]$$ are disjoint subsets of  $(2^kg(B)-(2^k-1)g(B))\cap (g(b_1),g(b_{|B|})]$ and contain at least $\mathcal{N}_{k-1}(\Delta_{d,b_{e_i}}g(B_d(i)))$ elements each, yielding the inequality.

% \begin{equation*}
%     \mathcal{N}_k(g(B))\gg \sum_{d \in D} \sum_{i=1}^{|B_d|} \mathcal{N}_{k-1}(\Delta_{d,b_{e_i}}g(B_d(i))).
% \end{equation*}

\end{proof}

\begin{proof}[Proof of Theorem \ref{betterKdependence}]
We will show that for any compact interval $I \subset \R$, any finite set $A \subset I$ and any pair of $k$-independent analytic functions $f,g$, that

$$\mathcal{N}_k(f(A))\cdot\mathcal{N}_k(g(A)) \gg_{f,g,k,I} \frac{|A|^{2k+1}}{K^{5\cdot2^{k-1}-2k-3}}.$$

    The proof is by induction on $k$, with base case Lemma \ref{basecase}. The inductive step goes as follows: Let $k\geq 2$ and suppose that the statement holds for all pairs of $(k-1)$-independent functions. Divide $I$ into $N=N(f,g,I)$ subintervals, with this value of $N$ coming from Lemma \ref{f'todeltaf}. By the pigeonhole principle, one of these subintervals contains $\Omega_{f,g,I}(|A|)$ elements of $A$. Call this subinterval $I'$. Since $k \geq 2$, the definition of $k$-independence tells us that neither $f$ nor $g$ can be linear. Thus, by analyticity of $f,g$ and a further application of the pigeonhole principle, there is a subinterval $I''\coloneqq[a,b] \subseteq I'$, containing $\Omega_{f,g,I}(|A|)$ points of $A$, upon which $f'$ and $g'$ are both strictly monotone and non-vanishing. Write $A'=A \cap I''$. We now replace $f$ and $g$ with $1$-convex functions $\hat{f}, \hat{g}$ that have the same additive properties. We define the function $\hat{f}$ as follows
\begin{align*}
    \hat{f}(x)&=\begin{cases}
        f(x) & \text{if } f \text{ and } f' \text{ are strictly increasing,}\\
        -f(x)& \text{if } f \text{ and } f' \text{ are strictly decreasing,}\\
        f(a+b-x)& \text{if } f \text{ is strictly decreasing and } f' \text{ is strictly increasing,}\\
        -f(a+b-x)&\text{if } f \text{ is strictly increasing and } f' \text{ is strictly decreasing.}\\
    \end{cases}
\end{align*}
It is easy to check that $\hat{f},\hat{f}'$ are strictly increasing, defined on $I$ and that $\mathcal{N}_s(\hat{f}(A'))=\mathcal{N}_s(f(A'))$ for all $s \in \mathbb{N}$. We define the function $\hat{g}$ similarly.

    We will now work in a large subset $A''$ of $A'$ in which each element $a_i$ satisfies 
$$n_{A'}(a_i,a_{i+1})  \ll \frac{n_{A'}(a_1,a_{|A'|})}{|A'|}.$$

By the pigeonhole principle, we have $|A''| = \Omega_{f,g,I}(|A|)$. Let $D$ be the set of consecutive differences for this set, noting that $d < \diam(I)/N$ for each $d \in D$. By Lemma \ref{equidistribution}, we have that $|D| = \Omega(|A+A-A||A|^{-1})$. We define, for each $d \in D$, the set $A''_d$ and its truncation $A''_d(i)=\{b_{e_1},\dots,b_{e_i}\}$ for each $i =1, \dots, |A_d|$ as in Lemma \ref{B_d}. One may replace each term in the following product using (\ref{squeezing2}) to find
    \begin{align*}
    \mathcal{N}_k(\hat f(A''))\cdot \mathcal{N}_k(\hat g(A'')) &\gg \left(\sum_{d \in D}\sum_{i=1}^{|A''_d|} \mathcal{N}_{k-1}(\Delta_{d,b_{e_i}}\hat f(A''_d(i)))\right)\cdot \left(\sum_{d \in D}\sum_{i=1}^{|A''_d|} \mathcal{N}_{k-1}(\Delta_{d,b_{e_i}}\hat g(A''_d(i)))\right)\\
    \end{align*}
    whereupon an application of the Cauchy-Schwarz inequality yields
    \begin{align*}
        \mathcal{N}_k(\hat f(A''))\cdot \mathcal{N}_k(\hat g(A'')) &\gg \left(\sum_{d \in D}\sum_{i=1}^{|A''_d|}\mathcal{N}^{1/2}_{k-1}(\Delta_{d,b_{e_i}}\hat f(A''_d(i)))\cdot \mathcal{N}^{1/2}_{k-1}(\Delta_{d,b_{e_i}}\hat g(A''_d(i)))\right)^2.\\ 
    \end{align*}
     By Lemma \ref{f'todeltaf}, $\Delta_{d,b_{e_i}}\hat f=\hat f(b_{e_i})+\Delta_d\hat f$ and $\Delta_{d,b_{e_i}}\hat g=\hat g(b_{e_i})+\Delta_d\hat g$ are $(k-1)$-independent and thus by the inductive hypothesis we have 
    \begin{align*}
    \mathcal{N}_k(\hat f(A''))\cdot \mathcal{N}_k(\hat g(A'')) &\gg_{f,g,k,I} \left(\sum_{d \in D}\sum_{i=1}^{|A''_d|} \left(\frac{|A''_d(i)|^{5\cdot 2^{k-2}-2}}{|A''_i(d)+A''_i(d)-A''_i(d)|^{5\cdot 2^{k-2}-2(k-1)-3}}\right)^{1/2} \right)^2.\\
    \end{align*}
    Noting that $|A''_d(i)|=i$ and $|A+A-A|\gg |A''_d(i)+A''_d(i)-A''_d(i)|$ and applying H\"{o}lder's inequality twice, we have. 
    \begin{align*}
   \mathcal{N}_k(\hat f(A''))\cdot \mathcal{N}_k(\hat g(A'')) &\gg_{f,g,k,I} \frac{1}{|A+A-A|^{5\cdot 2^{k-2}-2(k-1)-3}}\left(\sum_{d\in D}\sum_{i=1}^{|A''_d|} i^{5\cdot 2^{k-3}-1} \right)^2 \\ 
    &\gg \frac{1}{|A+A-A|^{5\cdot 2^{k-2}-2(k-1)-3}}\left(\sum_{d\in D}|A''_d|^{5\cdot 2^{k-3}} \right)^2\\
    &\gg_{f,g,k,I} \frac{1}{|A+A-A|^{5\cdot 2^{k-2}-2(k-1)-3}}\left(\frac{|A|^{5\cdot 2^{k-3}}}{|D|^{5\cdot 2^{k-3}-1}}\right)^2.
\end{align*}
Finally, noting that $|A+A-A|/|A| \ll |D|$ we have
\begin{align*}    
   \mathcal{N}_k(\hat f(A''))\cdot \mathcal{N}_k(\hat g(A'')) &\gg_{f,g,k,I} \frac{|A|^{5\cdot 2^{k-1}-2}}{|A+A-A|^{5\cdot 2^{k-1}-2k-3}} = \frac{|A|^{2k+1}}{K^{5\cdot 2^{k-1}-2k-3}}.
    \end{align*}
    The result follows upon noting that $\mathcal{N}_k(f(A)) \geq \mathcal{N}(\hat{f}(A''))$ and $\mathcal{N}_k(g(A)) \geq \mathcal{N}_k(\hat{g}(A''))$.
\end{proof}

% \begin{lemma}
%     \begin{lemma} Let $A$ be a set of real numbers and $f,g: \R \to \R$ be piece-wise monotone functions such that the map $\Psi_{f,g}:(x,y) \mapsto (f(x)-f(y),g(x)-g(y))$ is injective. Let $K \coloneqq |A+A-A|/|A|$, then 
%     $$\mathcal{N}^k_{f(A)}(\min f(A),\max f(A)) \cdot \mathcal{N}^k_{g(A)}(\min g(A),\max g(A)) \gg K^3|A|^2$$
% \end{lemma}
% \end{lemma}

% Theorem \ref{growthforsomef} follows from proving the following.

% \begin{theorem}\label{maintheorem}
% Let $A \subseteq \R$ and $\mathcal{F}$ be a set of $k+1$ 1-independent functions. Then there exists some $s \leq k$ such that at least one function $f \in \mathcal{F}$ must satisfy  
% $$\mathcal{N}_s(f(A)) \gg_k |A|^{\phi(k+1)}$$

% where $\phi:\mathbb{N} \to \R$ is the sequence satisfying the recursive formula
% \begin{align*}
% \phi(1)&=1,\\
%     \phi(k)&=1+\frac{1}{1+\frac{1}{\phi(k-1)}} \qquad \text{for all } k\geq 2.
% \end{align*}
% \end{theorem}

\begin{proof}[Proof of Theorem \ref{growthforsomef}]

% Writing $A=\{a_1<...<a_{|A|}\}$, there exists a large subset $I$ of $\{1,\dots,|A|\}$ for which the positive neighbouring differences $d_i=a_{i+1}-a_i$ form a set $D$ with $|D|\ll K$. Indeed, otherwise one could use squeezing for the complement $I$ to violate $|A+A-A|=K|A|.$ 
% We will in fact prove a stronger statement: that the same bound holds with each sum set $2^{k-j+1}f_j(A)-(2^{k-j+1}-1)f_j(A)$ intersected with the interval spanned by $f_j(A)$. To this end, we introduce the following notation: For any set $B \subseteq \R$ and $s \in \mathbb{N}$
% $$\mathcal{N}_s(B) \coloneqq |(2^sB-(2^s-1)B)\cap (\min B, \max B)|.$$
Book-keeping will be aided by indexing from $0$, whence let ${\mathcal{F}=\{f_i\}_{i=0}^k}$ be a set of $k+1$ 1-independent analytic functions. The stated theorem follows with $n=k+1$.  
We aim to prove the bound 
\begin{equation}\label{maxbound}
 \max_{1 \leq j \leq k}\{\mathcal{N}_k(f_0(A)), \mathcal{N}_{k-j+1}(f_j(A))\} \gg_{\mathcal{F},I} |A|^{\phi(k+1)}.   
\end{equation}
Recall that this implies
$$\max_{1 \leq j \leq k}\{|2^kf_0(A)-(2^k-1)f_0(A)|,|2^{k-j+1}f_j(A)-(2^{k-j+1}-1)f_j(A)|\} \gg_{\mathcal{F},I} |A|^{\phi(k)},$$
from which the stated result follows upon the observation that for any non-negative integers $t<s$, $|2^tf_j(A)-(2^t-1)f_j(A)|\leq |2^sf_j(A)-(2^s-1)f_j(A)|.$

The bound \eqref{maxbound} is a consequence of bounding the following product as so
\begin{equation}\label{prodbound}
\mathcal{N}_{k}(f_{0}(A)) \cdot \prod_{j=1}^{k}\mathcal{N}_{k-j+1}^{p(j)}(f_{j}(A)) \gg_{\mathcal{F},I} |A|^{p(k+1)}.
\end{equation}
 Here, $p:\mathbb{N} \cup \{0\} \to \mathbb{N}$ is the sequence given by the following recursive formula:
\begin{align*}
    &p(0)=p(1)=1,\\
   &p(j)=p(j-1)+\sum_{i=0}^{j-1} p(i) \qquad \text{ for all } j \geq 2.
\end{align*}
Indeed, this is the case as long as $\phi(k)=p(k)/\sum_{i=0}^{k-1} p(i)$ for all $k\geq 2$ which one verifies with the calculation
\begin{align*}
    \frac{p(k)}{\sum_{i=0}^{k-1} p(i)} &= 1+\frac{p(k-1)}{\sum_{i=0}^{k-1} p(i)}\\
    &=1+\frac{1}{\frac{\sum_{i=0}^{k-1} p(i)}{p(k-1)}}\\
    &=1+\frac{1}{1+\frac{\sum_{i=0}^{k-2}p(i)}{p(k-1)}}=1+\frac{1}{1+\frac{1}{\phi(k-1)}}
\end{align*}

and observation that $p(2)/(p(0)+p(1))=3/2=\phi(2)$.

The proof of (\ref{prodbound}) is by induction on $k\geq 1$ (note that the trivial $k=0$ case holds but is unsuitable for our induction). The base case is Lemma \ref{basecase}. We prove the inductive step as follows. Let $A$ be any finite subset of the compact interval $I$ and consider any set of $k+1$ 1-independent analytic functions $\mathcal{F}=\{f_i\}_{i=0}^k$ defined on $I$. Divide $I$ into $N=N(\mathcal{F},I)$ subintervals, with this value of $N$ coming from Lemma \ref{f'todeltaf2}. By the pigeonhole principle, one of these subintervals contains $\Omega_{\mathcal{F},I}(|A|)$ elements of $A$. Call this subinterval $I'$. By analyticity and another application of the pigeonhole principle, $I'$ contains a subinterval $I''$, containing $\Omega_{\mathcal{F},I}(|A|)$ elements of $A$, upon which each derivative $f_i'$ is non-vanishing. Whence each function $f_i$ is strictly monotone and hence invertible on $I''$. Write $A' = A \cap I''$, $h_i\coloneqq f_i \circ f_k^{-1}$ for each $0 \leq i \leq {k-1}$ and write $B=f_k(A')$. Note that none of the functions $h_i$ are linear. Indeed, for all $0 \leq i \leq k-1$, $h_i'=(f_i'/f_k')\circ f_k^{-1}$ is non-constant by 1-independence of $\mathcal{F}$. By analyticity and another application of the pigeonhole principle there exists a subinterval $I''\coloneqq [a,b] \subseteq I'$ containing $\Omega_{\mathcal{F},I}(|B|)$ elements of $B$ upon which the $h_i'$ are all strictly monotone. Write $B'=B \cap I''$.

We now replace the $h_i$ with $1$-convex functions that have the same additive properties. This is done as follows: for all $1 \leq i \leq k-1$ we define
\begin{equation*}
    g_i(x)=\begin{cases}
        h_i(x) & \text{if } h_i \text{ and }h_i' \text{ are strictly increasing,}\\
        -h_i(x) & \text{if } h_i \text{ and } h_i' \text{ are strictly decreasing,}\\
        h_i(a+b-x) & \text{if } h_i \text{ is strictly decreasing and } h_i' \text{ is strictly increasing,}\\
        -h_i(a+b-x) & \text{if } h_i \text{ is strictly increasing and } h_i' \text{ is strictly decreasing.}
    \end{cases}
\end{equation*}
It is easy to check, for all $1 \leq i \leq k-1$, that $g_i,g_i'$ are defined on $I$, are strictly increasing and that $\mathcal{N}_s(g_i(B'))=\mathcal{N}_s(h_i(B'))$ for all $s \in \mathbb{N}$.

As in the proof of Theorem \ref{betterKdependence}, we will work in a large subset $B'' \subseteq B'$ in which each element $b_i$ satisfies 
$$n_{B'}(b_i,b_{i+1})  \ll \frac{n_{B'}(b_1,b_{|B'|}))}{|B'|}.$$
By the pigeonhole principle, we have that $|B''| \gg |B|$. Let $D$ be the set of consecutive differences for this set and note that $d<\diam(I)/N$ for each $d \in D$. By Lemma \ref{equidistribution}, we have that $|D| \ll \mathcal{N}_1(B)|B|^{-1}$. We again define, for each $d \in D$, set $B''_d$ and its truncation $B''_d(i)=\{c_{e_1},\dots,c_{e_i}\}$ for each $1 \leq i \leq  |B_d|$ as in Lemma \ref{B_d}.

We now consider the product 
\begin{align*}
    P(k)&\coloneqq \mathcal{N}_{k}(g_{0}(B)) \cdot \prod_{j=1}^{k-1}\mathcal{N}_{k-j+1}^{p(j)}(g_{j}(B))\\
    & \geq \mathcal{N}_{k}(g_{0}(B'')) \cdot \prod_{j=1}^{k-1}\mathcal{N}_{k-j+1}^{p(j)}(g_{j}(B'')).
\end{align*}

One may replace each term in the above product using (\ref{squeezing2}) to find
\begin{align*}
    P(k)  & \gg \left(\sum_{d \in D}\sum_{i=1}^{|B'_d|} \mathcal{N}_{k-1}(\Delta_{d,c_{e_i}}g_0(B''_d(i)))\right)\cdot \prod_{j=1}^{k-1}\left(\sum_{d \in D}\sum_{i=1}^{|B''_d|} \mathcal{N}_{k-j}(\Delta_{d,c_{e_i}}g_j(B''_d(i)))\right)^{p(j)} 
    \end{align*}
    whereupon an application of H\"{o}lder's inequality, with $q(k)\coloneqq \sum_{i=0}^{k-1}p(i)$, yields
    \begin{align*}
    P(k) \gg \left(\sum_{d \in D}\sum_{i=1}^{|B''_d|} \left(\mathcal{N}^{p(0)}_{k-1}(\Delta_{d,c_{e_i}}g_0(B''_d(i)))  \cdot  \prod_{j=1}^{k-1}\mathcal{N}^{p(j)}_{k-j}(\Delta_{d,c_{e_i}}g_j(B''_d(i)))\right)^{1/{q(k)}}\right)^{q(k)}. 
    \end{align*}
    By Lemma \ref{f'todeltaf2} the set $\{\Delta_{d,c_{e_i}}g_i\}_{i=0}^{k-1}=\{g_i(c_{e_i})+\Delta_d(f_i \circ f_k^{-1})\}_{i=0}^{k-1}$ is a set of $k-1$ 1-independent functions. Hence, we may apply the inductive hypothesis to find that
    \begin{align*}
    P(k) & \gg_{\mathcal{F},I} \left( \sum_{d\in D} \sum_{i=1}^{|B''_d|} |B''_d(i)|^{p(k)/q(k)} \right)^{q(k)} \\
    & \gg \left( \sum_{d\in D} \sum_{i=1}^{|B''_d|} i^{p(k)/q(k)} \right)^{q(k)} \gg\left( \sum_{d\in D} |B''_d|^{p(k)/q(k)+1} \right)^{q(k)} \gg_{\mathcal{F},I} |B|^{p(k)+q(k)}|D|^{-p(k)}.
\end{align*}
% \begin{align*}
%     &\mathcal{N}_{g_{0}(B)}^{k}(g_{0}(B)) \cdot \prod_{j=1}^{k-1}(\mathcal{N}_{g_{j}(B)}^{k-j+1}(g_{j}(B)))^{p(j)} \\  
%     &\qquad \gg \left(\sum_{d \in D}\sum_{i=1}^{|B'_d|} \mathcal{N}_{\tilde{g}_0(B'_d(i))}^{k-1}(\tilde{g}_0(b_{e_1}),\tilde{g}_0(b_{e_i}))\right)\cdot \prod_{j=1}^{k-1}\left(\sum_{d \in D}\sum_{i=1}^{|B'_d|} \mathcal{N}_{\tilde{g}_j(B'_d(i))}^{k-j}(\tilde{g}_j(b_{e_1}),\tilde{g}_j(b_{e_i}))\right)^{p(j)} \\
%     &\qquad \gg \left(\sum_{d \in D}\sum_{i=1}^{|B'_d|} \left(\mathcal{N}_{\tilde{g}_0(B'_d(i))}^{k-1}(\tilde{g}_k(b_{e_1}),\tilde{g}_k(b_{e_i}))  \cdot  \prod_{j=1}^{k-1}\mathcal{N}_{\tilde{g}_j(B'_d(i))}^{k-j}(\tilde{g}_j(b_{e_1}),\tilde{g}(b_{e_i}))\right)^{1/{q(k)}}\right)^{q(k)} \\
%     &\qquad \gg \left( \sum_{d\in D} \sum_{i=1}^{|B'_d|} |B'_d(i)|^{p(k)/q(k)} \right)^{q(k)} \\
%     &\qquad \gg \left( \sum_{d\in D} \sum_{i=1}^{|B'_d|} i^{p(k)/q(k)} \right)^{q(k)}\\
%     &\qquad \gg \left( \sum_{d\in D} |B'_d|^{p(k)/q(k)+1} \right)^{q(k)}\\
%     &\qquad \gg |B|^{p(k)+q(k)}|D|^{-p(k)}
% \end{align*}
Since $|D|\ll \mathcal{N}_1(B)|B|^{-1}$ we have
$$P(k) \cdot \mathcal{N}_1(B)^{p(k)} \gg_{\mathcal{F},I} |B|^{2p(k)+q(k)}=|B|^{p(k+1)}.$$

Replacing $B$ with $f_{k}(A')$ and, for each $1 \leq i \leq k-1$, $g_i$ with $h_i=f_i\circ f_k^{-1}$ then yields

$$\mathcal{N}_{k+1}(f_{0}(A')) \cdot \prod_{j=1}^{k}(\mathcal{N}_{k-j+1}(f_{j}(A')))^{p(j)} \gg_{\mathcal{F},I} |A|^{p(k+1)}.$$

The desired result follows upon noting that 
$\mathcal{N}_k(f_j(A')) \leq \mathcal{N}_k(f_j(A))$ for all $0 \leq j \leq k$.

\end{proof}

% \textit{Remark.} Although we have used Lemma \ref{basecase} as the base case for this induction, note that its proof is essentially the inductive step above, using the $k=0$ case that $|f_0(A)|=|A|^{1}$ for any "$0$-convex" i.e. strictly monotone hence injective function $f_0$.

\begin{proof}[Proof of Theorem \ref{f(A-A)}] By analyticity of $f$, there is a subinterval of $I$ containing $\Omega_f(|A|)$ where $f$ and $f'$ are both strictly monotone. We will now assume that $f$ and $f'$ are both strictly increasing, noting that the other cases follow mutatis mutandis. 

    Consider an $(n+1)$-tuple $(a_0,\dots,a_n) \in A^{n+1}$ with the minimum possible length $a_{n}-a_{0}$. There are $\Omega_n(|A|)$ left translates $[a_{0},a_{n}]-h$ of the interval $[a_{0},a_{n}]$ by elements of $h \in A$ that do not intersect. Let $H \subset A$ be the set of these translates. For each $1 \leq l \leq n$ write $\delta_l=a_{l}-a_{0}$.  
    
    For each $1 \leq l \leq n$, consider the discrete derivatives
\begin{equation}\label{leftandright}
    \Delta_{\delta_l}f(x)\coloneqq f(x+\delta_i)-f(x).
\end{equation}

As $h$ increases, the value of $\Delta_{\delta_l}f(a_0-h)$ decreases, hence for $h_i<h_j$ one can squeeze, for any $1 \leq l \leq n$ the $n$-tuple $(f(a_0-h_j),f(a_1-h_j), \dots, f(a_{n}-h_j))$ into the $n$-tuple $(f(a_{0}-h_i),f(a_1-h_i),\dots,f(a_{n}-h_i))$ so that the leftmost points coincide and the $l$th point of the squeezed $n$-tuple lies to the left of the $l$th point of the larger interval. That is, there exists an $s_{i,j}$ such that $f(a_0-h_j)-s_{i,j}=f(a_0-h_i)$ and
\begin{align*}
    f(a_{l}-h_j)-s_{i,j} &\in (f(a_{0}-h_i),f(a_l-h_i))
\end{align*}
for each $1 \leq l \leq n$. Now define, for each $l,i$, 
$$T_l(i)\coloneqq \{f(a_l-h_j)-s_{i,j}:i<j\}$$
and notice that $$T_l(i)= f(a_0-h_i) + \Delta_{\delta_l}f(a_0-H(i))$$ with $H(i)$ denoting the set of the smallest $i-1$ elements of $H=\{h_1<\dots<h_{|H|}\}$.
Now, if the functions (\ref{leftandright}) are 1-independent, then by Theorem \ref{growthforsomef} we have that, for each $i$, some $T_{l'}(i) \in \{T_l(i):1 \leq l \leq n\} $ must satisfy
    $$\mathcal{N}_{n-1}(T_{l'}(i))) \gg_{f,n,I} i^{\phi(n)}.$$
By the squeezing lemma (Lemma \ref{squeezing lemma}), we have that, for all $i$
\begin{align*}
   T_{l'}(i) \subset (f(a_0-h_i),f(a_{l'}-h_i)).
\end{align*}
 The intervals in 
 \begin{align*}
     \{(f(a_0-h_i),f(a_{l'}-h_i)):1 \leq i \leq |H|\}
 \end{align*}
  are, respectively, disjoint by the definition of $H$. Note further that, for any $s \in \mathbb{N}$,
$$2^sT_{l'}(i)-(2^s-1)T_{l'}(i) \subseteq 2^{s+1}f(A-A)-(2^{s+1}-1)f(A-A).$$
Thus we have
$$|2^{n}f(A-A)-(2^{n}-1)f(A-A)| \gg \sum_{i=1}^{|H|}\mathcal{N}_{n-1}(T_{l'}(i))) \gg_{f,n,I} \sum_{i=1}^{|A|} i^{\phi(n)} \gg |A|^{1+\phi(n)}$$
as desired.
    % We must have, since $f$ is strictly increasing and strictly convex, that with $h_i>h_j$
    % $$f(a_l-h_j)-f(a_{l+1}-h_i)\leq f(a_l-h_j)-f(a_l-h_j)\leq f(a_{\frac{k+1}{2}}-h_j)-f(a_{\frac{k+1}{2}}-h_i)=s_{ij}$$
    % when $l\leq $
    % and
    %  $$f(a_l-h_j)-f(a_{l-1}-h_i)\geq f(a_{l-1}-h_j)-f(a_{l-1}-h_i) \geq f(a_{\frac{k+1}{2}}-h_j)-f(a_{\frac{k+1}{2}}-h_i)=s_{ij}$$
    %  when $l-1 \geq \frac{k+1}{2}$
    % As translates $(a,b,c)-h$ move leftward -- that is, as $h$ increases, the value of $\Delta^-_{d_0}f(a-h)$ increases. That is, for $h_i>h_j$ one can `squeeze' the triple $(f(a-h_j),f(b-h_j),f(c-h_j))$ into the triple $(f(a-h_i),f(b-h_i),f(c-h_i))$ such that the midpoints coincide, i.e. there is some $s_{i,j}$ such that $f(b-h_i)+s_{i,j}=f(b-h_j).$
    % Hence, we can squeeze $i-1$ triples into the $i$th triple. 
    % We set $$T_{1,i}=\tilde{f}(a-H(i))=f(a-h_i)+\Delta^-_df(a-H(i)) \subset f(a-h_i)+f(A-A)-f(A-A)$$
    % or instead write $$T_{1,i}=\{t_i-t_j':j < i\} = f(b-h_i)+\Delta^-_df(a-H(i))\subseteq f(b-h_i)+f(A-A)-f(A-A).$$
\end{proof}
\section{Proofs of Corollaries \ref{polynomials} and \ref{introangles}}\label{corollaries}
We first present the proof of Corollary \ref{polynomials}.
\begin{proof}[Proof of Corollary \ref{polynomials}]
The second derivative $f''$ of any polynomial $f$ of degree at least $2$ can vanish at only finitely many points, and upon removing these, $\R$ is partitioned into $O_f(1)$ subintervals on which $f''$ does not vanish. We work on whichever of these intervals is guaranteed by the pigeonhole principle to contain $\Omega_f(|A|)$ elements of $A$. Call such an interval $I$. Note that the non-vanishing of $f''$ guarantees the non-vanishing of $\Delta_\delta f'$ for any $0<\delta <\diam(I)$. Indeed, we have
$$\Delta_\delta f'(x) = \int_x^{x+\delta} f''(t)dt$$
for any $x \in I \cap(I-\delta)$, and the integral on the right cannot vanish as its integrand is continuous and does not vanish over the range of integration.

    Now we need only check for any polynomial $f$ of degree $m \geq n+1$ that, for any set of $n$ distinct real numbers $\{\delta_1,\dots, \delta_n\}$, the functions $\Delta_{\delta_i}f'(x)=f'(x+\delta_i)-f'(x)$ are linearly independent. Writing $f(x)=\sum_{j=0}^{m}c_jx^j$ we have that    
    \begin{align*}
      \Delta_{\delta_i}f'(x)=  f'(x+\delta_i)-f'(x)&=\left(\sum_{j=1}^{m}jc_j(x+\delta_i)^{j-1}\right)-\left(\sum_{j=1}^{m}jc_jx^{j-1}\right)\\
      &=\sum_{j=1}^{m}jc_j((x+\delta_i)^{j-1}-x^{j-1})\\
      &=\sum_{j=1}^{m}jc_j\left(-x^{j-1}+\sum_{l=0}^{j-1}\binom{j-1}{m}\delta_i^{j-l-1}x^l\right)\\
      &=\sum_{j=2}^{m}\sum_{l=0}^{j-2}\binom{j-1}{l}jc_j\delta_i^{j-l-1}x^l.
\end{align*}
Changing the order of summation, we have the following expression for the polynomial $\Delta_{\delta_i}f'$.
      \begin{align*}
      \Delta_{\delta_i}f'(x)&=\sum_{l=0}^{m-2}C_{l}(\delta_i)x^l \qquad \text{with} \qquad   C_l(\delta_i)\coloneqq\sum_{j=l+2}^{m}\binom{j-1}{l}jc_j\delta_i^{j-l-1}.
        \end{align*}
We now consider a linear combination of the $\Delta_{\delta_i}f'$ with constant coefficients $a_i$.
\begin{align*}
    \sum_{i=1}^{m}a_i\Delta_{\delta_i}f'(x)&=\sum_{i=1}^{n}a_i\sum_{l=0}^{m-2}C_lx^l=\sum_{l=0}^{m-2}x^l\sum_{i=1}^na_iC_l(\delta_i).
\end{align*}
This linear combination vanishes everywhere if and only if, for each $0 \leq l \leq m-2$, we have that $\sum_{i=1}^na_iC_l(\delta_i)=0$. Suppose that this is the case. We then have the following system of equations
\begin{align*}
    \binom{m-1}{m-2}mc_{m}\sum_{i=1}^na_i\delta_i&=0,\\
    \binom{m-2}{m-3}(m-1)c_{m-1}\sum_{i=1}^na_i\delta_i+\binom{m-1}{m-3}mc_{m}\sum_{i=1}^na_i\delta_i^2&=0,\\
    \binom{m-3}{m-4}(m-2)c_{m-2}\sum_{i=1}^na_i\delta_i+\dots+\binom{m-1}{m-4}mc_{m}\sum_{i=1}^na_i\delta_i^3&=0,\\
    &\vdots\\
    \binom{1}{0}2c_2\sum_{i=1}^na_i\delta_i+\dots+\binom{m-1}{0}mc_{m}\sum_{i=1}^na_i\delta_i^{m-1} &=0.
\end{align*}
From considering these equations in order, one deduces that $\sum_{i=1}^na_i\delta_i^s=0$ for all $s\leq m$. That is, the vector $\mathbf{a}=(a_1,\dots,a_n)$ lies in the kernel of the following $(m-1)\times n$ matrix $$M\coloneqq\begin{pmatrix}
    \delta_1 & \delta_2 & \dots & \delta_n\\
    \delta_1^2 &\delta_2^2 &\dots & \delta_n^2\\
    \vdots & \vdots & \ddots &\vdots \\
    \delta_1^{m-1} & \delta_2^{m-1} & \dots & \delta_n^{m-1}
\end{pmatrix}=V_{m-1}(\delta_1,\dots,\delta_n)^T\cdot\diag(\delta_1,\dots,\delta_n).$$
% \begin{pmatrix}
%     1 &1 &\dots &1\\
%     \delta_1 & \delta_2 &\dots &\delta_k\\
%     \vdots &\vdots &\ddots &\vdots\\
%     \delta_1^{n-1} & \delta_2^{n-1}&\dots & \delta_k^{n-1}
% \end{pmatrix}\begin{pmatrix}
%     \delta_1 & & & \\
%     &\delta_2 & &\\
%     &&\ddots&\\
%     &&&\delta_n
% \end{pmatrix}
Where $V_{m-1}(\delta_1,\dots,\delta_n)$ is a $n \times {(m-1)}$ Vandermonde matrix. Taking $M_n$ to be the $n\times n$ submatrix formed from $M$ by removing the bottom $m-n$ rows, we have
$$\det M_n =\det V_{n}(\delta_i,\dots,\delta_n) \cdot \det \diag(\delta_1,\dots,\delta_n) = \left(\prod_{1 \leq i < j \leq n}(\delta_i-\delta_j)\right)\cdot\left(\prod_{i=1}^n \delta_{i}\right).$$
Which we know must not vanish since each $\delta_i$ is distinct. Whence $\rank (M) \geq \rank (M_n) =n$, and the kernel of $M$ must be trivial, i.e. $\mathbf{a}=\mathbf{0}$. Thus, the functions $\Delta_{\delta_i}f'$ must be linearly independent.

% The coefficient of $x^{k-1}$ is 
% $$0=\sum_{i=1}^ka_iC_{k-1}(\delta_i)=\binom{k}{k-1}(k+1)c_{k+1}\sum_{i=1}^ka_i\delta_i.$$
% Hence we know that $\sum_{i=1}^ka_i\delta_i=0.$ We now consider the coefficient of $x^{k-2}$:
% $$$$
   % \sum_{j=l+2}^{k+1}\binom{j-1}{l}jC_j\delta_i^{j-l-1}\\
   %  &=\sum_{l=0}^{k-1}x^l\sum_{j=l+2}^{k+1}\binom{j-1}{l}jC_j\sum_{i=1}^{k}a_i\delta_i^{j-l-1}% 
    
%     Since $f$ has degree $k+1$, its leading term is of the form $Cx^{k+1}$ for some constant $C$. Thus 
%     \begin{align*}
%     \Delta_{\delta_i}f^{(k-1)}(x)=f^{(k-1)}(x+\delta_i)-f^{(k-1)}(x)&=C\frac{(k+1)!}{2}\cdot(x+\delta_i)^2-C\frac{(k+1)!}2\cdot x^2\\&=C\frac{(k+1)!}{2}\cdot (2\delta_ix+\delta_i^2),
%     \end{align*}
%     for $i=1,2$. Whence $\Delta_{\delta_i} f^{(k)}(x)=C(k+1)!\cdot\delta_i$ for $i=1,2$. Thus, computing the Wro\'{n}skian, we have 
% \begin{align*}
%     W(\Delta_{\delta_1}f,\Delta_{\delta_2}f)(x)&=\det\begin{pmatrix}
%         \Delta_{\delta_1}f^{(k-1)}(x) & \Delta_{\delta_2}f^{(k-1)}(x)\\
%         \Delta_{\delta_1}f^{(k)}(x) & \Delta_{\delta_2}f^{(k)}(x)
%     \end{pmatrix}\\&=\det\begin{pmatrix}
%        C\frac{(k+1)!}{2}\cdot (2\delta_1x+\delta_1^2)& C\frac{(k+1)!}{2}\cdot(2\delta_2x+\delta_2^2)\\
%         C(k+1)!\cdot\delta_1&C(k+1)!\cdot\delta_2
%     \end{pmatrix}\\
%     &=C\frac{((k+1)!)^2}{2}\cdot(\delta_1^2\delta_2-\delta_1\delta_2^2) \neq 0
% \end{align*}
% And hence $\Delta_{\delta_1}f$ and $\Delta_{\delta_2}f$ are $(k-1)$-independent.
\end{proof}

We now present the proof of Corollary \ref{introangles}, following roughly the proof of \cite[Theorem 1]{Olly}

\begin{proof}[Proof of Corollary \ref{introangles}] Since the set of angles is invariant under translation, rotation, reflection and scaling, we may, by a combination of the above transformations, position our set $A\times A$ so that $p=(0,0)$ and a subset $A'$ of at least $|A|/2$ elements of $A$ lies in the interval $[0,1]$. We now work with $A'\times A' \subset [0,1]\times[0,1]$. Set $A''=\log(A'\setminus\{0\})$, and $f(x)=\arctan(e^x)$ in Theorem \ref{f(A-A)}. Then $f(A''-A'')=\arctan\left(\frac{A'\setminus\{0\}}{A'\setminus \{0\}}\right)$. Notice that the angle at $(0,0)$ between points $(x,y),(0,0),(x',y') \in A' \times A'$ is given by 
$$\arctan\left(\frac{y}{x}\right)-\arctan\left(\frac{y'}{x'}\right).$$
Thus $f(A''-A'')-f(A''-A'') \subseteq \mathcal{A}(A \times A)$. 
By Theorem \ref{f(A-A)} with $n=3$ we have 
\begin{align*}
 |8\mathcal{A}(A\times A)|&=|8f(A''-A'')-8f(A''-A'')|\\ &\geq |8f(A''-A'')-7f(A''-A'')| \gg |A''|^{1+\phi(3)}= |A|^{13/5}.
\end{align*}
It just remains to check that for any distinct $\delta_1,\delta_2,\delta_3$ that the functions
$\Delta_{\delta_i}f$ are 1-independent i.e. that the functions $\Delta_{\delta_i}f'$ are linearly independent. 

Use of a computer algebra system (Mathematica \cite{Mathematica}, in our case) and the chain rule gives us the following expression for the Wro\'{n}skian $W(\Delta_{\delta_1}f',\Delta_{\delta_2}f',\Delta_{\delta_3}f')$:

\begin{align}\label{arctanWronskian}
     W(\Delta_{\delta_1}f',\Delta_{\delta_2}f',\Delta_{\delta_3}f')(x)=\frac{C_1y^4+C_2y^6+C_3y^8+C_4y^{10}}{(1+y^2)^{3}(1+\sigma_1^2y^2)^{3}(1+\sigma_2^2y^2)^{3}(1+\sigma_3^2y^2)^{3}}.
\end{align}
Where $y(x),\sigma_1,\sigma_2,\sigma_3,C_1,C_2,C_3,C_4$ are defined as follows:
\begin{align*}
    y(x) &= e^x,\\
    \sigma_i&=e^{\delta_i}, \quad i=1,2,3,\\
    C_1&=-16 (-1+{\sigma_1}) ({\sigma_1}-{\sigma_2}) (-1+{\sigma_2}) ({\sigma_1}-{\sigma_3}) ({\sigma_2}-{\sigma_3}) (-1+{\sigma_3})\\ &\qquad \cdot (-{\sigma_1}-{\sigma_1}^2-{\sigma_2}-2 {\sigma_1} {\sigma_2}-{\sigma_1}^2 {\sigma_2}-{\sigma_2}^2-{\sigma_1} {\sigma_2}^2-{\sigma_3}-\\ &\qquad 2 {\sigma_1} {\sigma_3}-{\sigma_1}^2 {\sigma_3}-2 {\sigma_2} {\sigma_3}-2 {\sigma_1} {\sigma_2} {\sigma_3}-{\sigma_2}^2 {\sigma_3}-{\sigma_3}^2-{\sigma_1} {\sigma_3}^2-{\sigma_2} {\sigma_3}^2),\\
    C_2 &=  -16 (-1+{\sigma_1}) ({\sigma_1}-{\sigma_2}) (-1+{\sigma_2}) ({\sigma_1}-{\sigma_3}) ({\sigma_2}-{\sigma_3}) \\&\qquad \cdot (3 {\sigma_1} {\sigma_2} {\sigma_3}+3 {\sigma_1}^2 {\sigma_2} {\sigma_3}+3 {\sigma_1} {\sigma_2}^2 {\sigma_3}+3 {\sigma_1} {\sigma_2} {\sigma_3}^2),\\
    C_3&= -16 (-1+{\sigma_1}) ({\sigma_1}-{\sigma_2}) (-1+{\sigma_2}) ({\sigma_1}-{\sigma_3}) ({\sigma_2}-{\sigma_3}) (-1+{\sigma_3}) \\&\qquad \cdot (-3 {\sigma_1}^2 {\sigma_2}^2 {\sigma_3}-3 {\sigma_1}^2 {\sigma_2} {\sigma_3}^2-3 {\sigma_1} {\sigma_2}^2 {\sigma_3}^2-3 {\sigma_1}^2 {\sigma_2}^2 {\sigma_3}^2),\\
    C_4 &=  -16 (-1+{\sigma_1}) ({\sigma_1}-{\sigma_2}) (-1+{\sigma_2}) ({\sigma_1}-{\sigma_3}) ({\sigma_2}-{\sigma_3}) (-1+{\sigma_3})\\ &\qquad \cdot ({\sigma_1}^3 {\sigma_2}^2 {\sigma_3}+{\sigma_1}^2 {\sigma_2}^3 {\sigma_3}+{\sigma_1}^3 {\sigma_2}^3 {\sigma_3}+{\sigma_1}^3 {\sigma_2} {\sigma_3}^2+2 {\sigma_1}^2 {\sigma_2}^2 {\sigma_3}^2+2 {\sigma_1}^3 {\sigma_2}^2 {\sigma_3}^2\\ &\qquad+{\sigma_1} {\sigma_2}^3 {\sigma_3}^2+2 {\sigma_1}^2 {\sigma_2}^3 {\sigma_3}^2 +{\sigma_1}^3 {\sigma_2}^3 {\sigma_3}^2+{\sigma_1}^2 {\sigma_2} {\sigma_3}^3+{\sigma_1}^3 {\sigma_2} {\sigma_3}^3+{\sigma_1} {\sigma_2}^2 {\sigma_3}^3\\ &\qquad+2 {\sigma_1}^2 {\sigma_2}^2 {\sigma_3}^3+{\sigma_1}^3 {\sigma_2}^2 {\sigma_3}^3+{\sigma_1} {\sigma_2}^3 {\sigma_3}^3+{\sigma_1}^2 {\sigma_2}^3 {\sigma_3}^3).
\end{align*}
By inspecting the above expressions and noting that the $\sigma_i$ are all distinct and greater than $1$, we see that $C_1,C_2,C_3$ and $C_4$ are all non-zero constants. Hence, the polynomial in the numerator on the right-hand side of \eqref{arctanWronskian} cannot vanish everywhere. This proves linear independence of the functions $\Delta_{\delta_i}f'$.
\end{proof}
\bibliography{bib}{}
\bibliographystyle{plain}
\end{document}